\documentclass[11pt]{amsart}

\usepackage{amsmath,amssymb,amsfonts,amsthm,mathtools}
\usepackage{mathrsfs}
\usepackage{enumitem}
\usepackage{array}
\usepackage{hyperref}
\hypersetup{
    colorlinks=true,
    linkcolor=blue,
    citecolor=red,
    urlcolor=red
}

\theoremstyle{plain}
\newtheorem{thm}{Theorem}[section]
\newtheorem{prop}[thm]{Proposition}
\newtheorem{lem}[thm]{Lemma}
\newtheorem{cor}[thm]{Corollary}

\theoremstyle{definition}
\newtheorem{dfn}[thm]{Definition}
\newtheorem{example}[thm]{Example}

\theoremstyle{remark}
\newtheorem{rem}[thm]{Remark}

\newcommand{\C}{\mathbb{C}}
\newcommand{\R}{\mathbb{R}}
\newcommand{\Z}{\mathbb{Z}}

\DeclareMathOperator{\Fred}{Fred}
\DeclareMathOperator{\Ker}{Ker}
\DeclareMathOperator{\Coker}{Coker}

\newcommand{\Fredzero}{\Fred_{0}}
\newcommand{\Fredsaone}{\Fred_{\mathrm{sa}}^{1}}

\newcommand{\IntroThm}[1]{\hyperref[#1]{Theorem~\ref*{#1}}}
\newcommand{\LaterThm}[1]{\hyperref[#1]{Theorem~\ref*{#1}}}
\newcommand{\LaterProp}[1]{\hyperref[#1]{Proposition~\ref*{#1}}}
\newcommand{\LaterDfn}[1]{\hyperref[#1]{Definition~\ref*{#1}}}
\newcommand{\LaterLem}[1]{\hyperref[#1]{Lemma~\ref*{#1}}}
\newcommand{\LaterCor}[1]{\hyperref[#1]{Corollary~\ref*{#1}}}
\newcommand{\LaterRem}[1]{\hyperref[#1]{Remark~\ref*{#1}}}
\newcommand{\LaterStmt}[1]{\hyperref[#1]{\ref*{#1}}}

\title[Odd Koschorke Classes]
{Odd Koschorke Classes}

\author[K. Horie]{Kyouhei Horie}

\address{
Department of Mathematics, Institute of Science Tokyo\\
2-12-1 Ookayama, Meguro-ku, Tokyo 152-8551, Japan
}

\email{horie.k.1e2d@m.isct.ac.jp}

\date{}

\begin{document}

\begin{abstract}
We introduce odd Koschorke classes in odd \(K\)-theory by using degeneracy loci
of self-adjoint Fredholm operators. These classes are characteristic classes
analogous to the even Koschorke classes in even \(K\)-theory. We study two aspects of these
classes: their role as obstruction classes and their realization as
characteristic classes with real coefficients for odd twisted \(K\)-theory.  On
the even side, we introduce generalized Koschorke classes indexed by arbitrary
partitions via Cibotaru's notion of a quasi-manifold.  These classes form a
\(\C\)-basis of \(H^*(\Fred_0;\C)\) and recover the usual Koschorke classes for
rectangular partitions. Finally, by analogy with the correspondence between even Koschorke classes
and singular vectors in a representation of the Virasoro algebra, we state a
result about a representation of a super-Virasoro algebra: each singular vector
in that representation has a unique finite expansion in terms of generalized Koschorke classes and
generalized odd Koschorke classes.
\end{abstract}

\maketitle
\tableofcontents

%%%%%%%%%%%%%%%%%%%%%%%%%%%%%%%%%%%%%%%%%%%%%%%%

\section{Introduction}
%%%%%%%%%%%%%%%%%%%%%%%%%%%%%%%%%%%%%%%%%%%%%%%%

\subsection{Background}

Topological \(K\)-theory is a generalized cohomology theory which assigns to a
compact Hausdorff space \(X\) a group built from complex vector bundles on
\(X\); see Atiyah's book \cite{AtiyahKTheory}.  One classical way to study
characteristic classes in topological \(K\)-theory is to realize them by
geometric degeneracy loci.  In the finite-dimensional theory, Chern classes and
Schur polynomial classes are represented by Schubert cycles on Grassmannians;
see, for example, Fulton \cite[Chapter~9]{Fulton}.

In the infinite-dimensional setting, Koschorke developed an analogous theory for
Fredholm bundle maps \cite{Koschorke}.  The basic idea is that the set of
Fredholm operators with prescribed dimensions of kernel and cokernel is a
stratum of finite codimension in the space of Fredholm operators.  The cohomology class associated with this stratum is a universal cohomology class;
its pullback along a family of Fredholm operators is the corresponding Koschorke class.  Atiyah and Jones used
Koschorke classes in their study of Yang--Mills theory: the Koschorke classes
for the Dirac operators over the space of connections give information on the
existence of connections for which the Dirac equation has a large solution
space \cite{AtiyahJones}.
Morimoto studied infinite-dimensional cycles associated with operators
\cite{Morimoto}.  From the present point of view, Morimoto's work belongs to the
same general strategy: cohomology classes of operator spaces are represented by
geometric subsets defined by degeneracy conditions.

Thus there is a reasonable amount of information about the even Koschorke
classes.  By contrast, the corresponding odd theory has not been developed to
the same extent.  In particular, there seems to be little literature on
classes corresponding to odd \(K\)-theory in the sense of Koschorke.  The main purpose of
this paper is to construct and study such classes, which we call \emph{odd
Koschorke classes}.

\subsection{Even and odd Koschorke classes}

We first recall the basic properties of the even Koschorke classes.  Let
\(\mathcal H\) be a separable infinite dimensional complex Hilbert space, and let
\(\Fred(\mathcal H)\) be the space of bounded Fredholm operators on
\(\mathcal H\).  For an integer \(m\), denote by \(\Fred_m\) the connected
component consisting of Fredholm operators of index \(m\).  In this component we
put
\[
A_{p,q}=
\{T\in\Fred_{p-q}\mid
\dim_{\C}\Ker T=p,\ \dim_{\C}\Coker T=q\}.
\]
The cohomology class associated with this stratum is the universal Koschorke class
\[
k_{p,q}\in H^{2pq}(\Fred_{p-q}).
\]
For a continuous family \(A\colon X\to\Fred_{p-q}\), we write
\[
\chi_{p,q}(A)=A^*k_{p,q}\in H^{2pq}(X).
\]

\begin{thm}[Even Koschorke classes
(Propositions~\ref{prop:even-vanishing}, \ref{prop:even-obstruction};
Theorem~\ref{thm:even-koschorke-determinant})]
\label{thm:intro-even-koschorke}
Let \(p,q\geq 1\), and let \(A\colon X\to\Fred_{p-q}\) be a continuous family.
The following assertions hold.
\begin{enumerate}[label=\textup{(\roman*)}]
\item If \(\dim_{\C}\Ker A_x<p\) and \(\dim_{\C}\Coker A_x<q\)
for every \(x\in X\), then
\[
\chi_{p,q}(A)=0.
\]
\item If \(X\) is a CW complex and \(\chi_{p,q}(A)=0\), then
\(A|_{X^{(2pq)}}\) is homotopic to a continuous map
\[
B\colon X^{(2pq)}\longrightarrow \Fred_{p-q}
\]
such that
\[
\dim_{\C}\Ker B_x<p
\qquad\text{and}\qquad
\dim_{\C}\Coker B_x<q
\]
for every \(x\in X^{(2pq)}\).  Here \(X^{(2pq)}\) denotes the
\(2pq\)-skeleton of the CW complex \(X\).
\item Let \(c_i\in H^{2i}(\Fred_{p-q})\) denote the universal Chern
classes.  Then one has the Hankel determinant formula
\[
k_{p,q}=\det(c_{p-i+j})_{1\leq i,j\leq q}.
\]
\end{enumerate}
\end{thm}

The odd Koschorke classes constructed in this paper are the self-adjoint
counterparts of these classes.  Let \(\Fredsaone\) denote the non-contractible
component of the space of bounded self-adjoint Fredholm operators.  By the work of Atiyah and Singer \cite{AtiyahSingerSkew}, this space
is a model for odd complex \(K\)-theory.
For \(p\geq0\), put
\[
A_p^{\mathrm{sa}}
=
\{A\in\Fredsaone\mid \dim_{\C}\Ker A=p\}.
\]
From the cohomology class associated with this stratum, we define the universal
odd Koschorke class
\[
k_p^{\mathrm{odd}}\in H^{p^2}(\Fredsaone).
\]
For a continuous family \(A\colon X\to\Fredsaone\), we set
\[
\chi_p^{\mathrm{odd}}(A)=A^*k_p^{\mathrm{odd}}\in H^{p^2}(X).
\]

\begin{thm}[Odd Koschorke classes
(Definitions~\ref{dfn:odd-universal-koschorke},
\ref{dfn:odd-koschorke-characteristic}; Proposition~\ref{prop:odd-vanishing};
Theorem~\ref{thm:odd-koschorke-obstruction};
Corollary~\ref{cor:odd-koschorke-avoidance};
Theorem~\ref{thm:odd-koschorke-chern})]
\label{thm:intro-universal-odd-koschorke}
\label{thm:intro-odd-vanishing}
\label{thm:intro-odd-obstruction}
\label{thm:intro-odd-chern-formula}
Let \(p\geq1\), and let \(A\colon X\to\Fredsaone\) be a continuous family.  The
following assertions hold.
\begin{enumerate}[label=\textup{(\roman*)}]
\item If \(\dim_{\C}\Ker A_x<p\) for every \(x\in X\), then
\[
\chi_p^{\mathrm{odd}}(A)=0.
\]
\item If \(X\) is a CW complex and \(\chi_p^{\mathrm{odd}}(A)=0\), then
\(A|_{X^{(p^2)}}\) is homotopic to a continuous map
\[
B\colon X^{(p^2)}\longrightarrow \Fredsaone
\]
such that
\[
\dim_{\C}\Ker B_x<p
\]
for every \(x\in X^{(p^2)}\).  Here \(X^{(p^2)}\) denotes the
\(p^2\)-skeleton of the CW complex \(X\).
\item Let \(c_{(2m+1)/2}\in H^{2m+1}(\Fredsaone)\) denote the universal odd
Chern classes, recalled in Section~3.5.  Then one has
\[
k_p^{\mathrm{odd}}
=
c_{1/2}c_{3/2}\cdots c_{(2p-1)/2}.
\]
\end{enumerate}
\end{thm}

The construction of the odd Koschorke classes and the proofs of assertions
(i) and (ii) of Theorem~\ref{thm:intro-universal-odd-koschorke} are formally
parallel to the even case.  The proof of the formula in (iii), which expresses the class in terms of Chern classes, follows
a different route from the proof in the even case and uses Cibotaru's notion of a
quasi-manifold \cite{Cibotaru}.

The following table compares the even side and the odd side.

\begin{center}
\renewcommand{\arraystretch}{1.25}
\begin{tabular}{|>{\raggedright\arraybackslash}p{0.18\textwidth}|>{\centering\arraybackslash}p{0.34\textwidth}|>{\centering\arraybackslash}p{0.34\textwidth}|}
\hline
& \textbf{Even side} & \textbf{Odd side} \\
\hline
Classifying space
& \(\Fred_{p-q}\)
& \(\Fredsaone\) \\
\hline
Degeneracy stratum
& \(A_{p,q}=\{T\mid \dim_\C\Ker T=p,\ \dim_\C\Coker T=q\}\)
& \(A_p^{\mathrm{sa}}=\{A\mid \dim_\C\Ker A=p\}\) \\
\hline
Real codimension
& \(\operatorname{codim}_{\R}(A_{p,q})=2pq\)
& \(\operatorname{codim}_{\R}(A_p^{\mathrm{sa}})=p^2\) \\
\hline
Universal class
& \(k_{p,q}\)
& \(k_p^{\mathrm{odd}}\) \\
\hline
Formula in terms of Chern classes
& \(k_{p,q}=\det(c_{p-i+j})_{1\leq i,j\leq q}\)
& \(k_p^{\mathrm{odd}}=c_{1/2}c_{3/2}\cdots c_{(2p-1)/2}\) \\
\hline
Avoidance condition
& \(\dim_\C\Ker T<p\) and \(\dim_\C\Coker T<q\)
& \(\dim_\C\Ker A<p\) \\
\hline
\end{tabular}
\end{center}

\subsection{Twisted \(K\)-theory}

Twisted \(K\)-theory is topological \(K\)-theory with a form of local
coefficients.  We use the formulation of Atiyah and Segal
\cite{AtiyahSegalTwisted}.  Atiyah and Segal also studied characteristic classes
for twisted \(K\)-theory with real coefficients and gave a condition which
decides when an ordinary characteristic class in \(K\)-theory lifts to a
characteristic class in twisted \(K\)-theory \cite{AtiyahSegal}.  The even
Koschorke classes satisfy this condition.

It is therefore natural to ask whether the odd Koschorke classes lift to twisted
\(K\)-theory.  Gomi studied conditions which decide when characteristic classes
in odd topological \(K\)-theory lift to characteristic classes for odd twisted
\(K\)-theory \cite{GomiMickelsson}.  Using this result, we prove that the odd
Koschorke classes lift to characteristic classes for odd twisted \(K\)-theory
(Theorem~\ref{thm:odd-koschorke-twisted-characteristic}).

\subsection{The relation with the Virasoro algebra}

Atiyah and Segal implicitly suggested that Koschorke classes are related to singular
vectors in a representation of the Virasoro algebra \cite{AtiyahSegal}.  Gomi
gave an explicit description of this relation \cite{GomiSugaku}: under the
realization of \(H^*(\Fred_0;\C)\) as a bosonic Fock space, the Koschorke classes correspond to singular vectors in the corresponding Virasoro representation.

After the odd Koschorke classes are added, it is natural to ask whether an
analogous statement holds for the super-Virasoro algebra.  The answer obtained in
this paper is not a perfectly parallel analogue of the even statement.  To
formulate the result, one has to generalize both the Koschorke classes and the
odd Koschorke classes.  In this paper we use Cibotaru's notion of a
quasi-manifold \cite{Cibotaru} to define generalized Koschorke classes and
generalized odd Koschorke classes indexed by partitions.  With these classes, we
prove that every singular vector in the \(N=1\) Neveu--Schwarz representation
considered in Section~5 has a unique finite
expansion with respect to the basis formed by the corresponding tensor products
(Theorem~\ref{thm:super-singular-expansion-paper}).

\subsection{Organization}

The paper is organized as follows. In Section~2 we recall the even
Koschorke classes, including their interpretation as characteristic classes,
their meaning from the viewpoint of obstruction theory, and their expression in terms of Chern classes.
In Section~3 we construct odd Koschorke classes using the model formed by self-adjoint
Fredholm operators, establish their basic properties, describe their interpretation from the viewpoint of obstruction theory, and explain their relation to twisted \(K\)-theory with real
coefficients. In Section~4 we introduce generalized Koschorke classes indexed
by arbitrary partitions, prove the Giambelli formula, explain how rectangular
partitions recover the even Koschorke classes, and recall the generalized
odd Koschorke classes needed to prove the formula expressing the universal odd Koschorke classes in terms of
odd Chern classes. In Section~5 we discuss the relation between Koschorke classes and the
Virasoro algebra, and we show that singular vectors in the Neveu--Schwarz representation considered in this paper
can be expanded using generalized Koschorke classes and generalized odd
Koschorke classes.

Unless otherwise specified, all cohomology groups are singular cohomology
groups with integer coefficients. We specify when real or complex coefficients
are used.

\subsection*{Acknowledgements.}

The author would like to thank his supervisor, Kiyonori Gomi. This work was supported by JST SPRING, Japan Grant Number JPMJSP2180, and by the Science Tokyo Support Program for Doctoral Students, funded by the Universities for International Research Excellence.

%%%%%%%%%%%%%%%%%%%%%%%%%%%%%%%%%%%%%%%%%%%%%%%%

\section{Koschorke classes}
%%%%%%%%%%%%%%%%%%%%%%%%%%%%%%%%%%%%%%%%%%%%%%%%

\subsection{The space of Fredholm operators and even \(K\)-theory}

Let $\mathcal{H}$ be a separable infinite-dimensional complex Hilbert space.
We write $\Fred(\mathcal{H})$ for the space of bounded Fredholm operators on
$\mathcal{H}$ endowed with the operator norm topology, and for each
$p\in \Z$ we write
\[
\Fred_{p}(\mathcal{H})
=
\{T\in \Fred(\mathcal{H})\mid
\operatorname{index}(T)=\dim\Ker T-\dim\Coker T=p\}.
\]
In particular,
\[
\Fred(\mathcal{H})
=
\coprod_{p\in \Z}\Fred_{p}(\mathcal{H}).
\]

Throughout this paper, for a paracompact Hausdorff space $X$, we adopt the
Fredholm-operator description as the definition of the even $K$-groups:
\[
K(X)=[X,\Fred(\mathcal{H})],
\qquad
\widetilde{K}(X)=[X,\Fred_{0}(\mathcal{H})].
\]
Here $[X,Y]$ denotes the set of homotopy classes of continuous maps from $X$ to
$Y$.

To define the addition operation on $K(X)$, we use two standard facts about separable
infinite-dimensional complex Hilbert spaces.
First, the orthogonal direct sum of two such Hilbert spaces is again a separable
infinite-dimensional complex Hilbert space.
Second, separable infinite-dimensional complex Hilbert spaces are unique up to
unitary isomorphism.
Hence there exists a unitary isomorphism
\[
\mathcal{H}\oplus \mathcal{H}\cong \mathcal{H}.
\]
We fix such an isomorphism once and for all.
Given $S,T\in \Fred(\mathcal{H})$, we form their direct sum
\[
S\oplus T\in \Fred(\mathcal{H}\oplus\mathcal{H}),
\]
and, via the above fixed unitary isomorphism, we regard $S\oplus T$ as an
element of $\Fred(\mathcal{H})$.
This induces an operation on homotopy classes
\[
[X,\Fred(\mathcal{H})]\times [X,\Fred(\mathcal{H})]
\longrightarrow
[X,\Fred(\mathcal{H})],
\]
and similarly on $[X,\Fred_{0}(\mathcal{H})]$.

The resulting operation is independent of the chosen unitary identification
$\mathcal{H}\oplus\mathcal{H}\cong\mathcal{H}$ up to canonical homotopy.  This
uses Kuiper's theorem, which states that the unitary group of a separable
infinite-dimensional Hilbert space is contractible in the norm topology
\cite{Kuiper}.  Hence the group law is well defined on homotopy classes.

From now on, we suppress the Hilbert space $\mathcal{H}$ in the notation and
write simply
\[
\Fred=\Fred(\mathcal{H}),
\qquad
\Fred_{p}=\Fred_{p}(\mathcal{H}),
\qquad
\Fred_{0}=\Fred_{0}(\mathcal{H}).
\]

Let $A$ be a commutative ring.  We write $H^{*}(X;A)$ for the
ordinary cohomology of a paracompact Hausdorff space $X$ with coefficients in $A$.

\begin{dfn}
A \emph{characteristic class of even $K$-theory with coefficients in $A$} is a
natural transformation
\[
\Phi\colon K(-)\longrightarrow H^{*}(-;A)
\]
from the functor \(K(-)\) on paracompact Hausdorff spaces to ordinary cohomology
with coefficients in $A$.  Similarly, a \emph{characteristic class of reduced
$K$-theory with coefficients in $A$} is a natural transformation
\[
\widetilde{\Phi}\colon \widetilde{K}(-)\longrightarrow H^{*}(-;A).
\]
\end{dfn}

\subsection{Even degeneracy strata}

Sections~2.2--2.5 recall Koschorke's results on Fredholm degeneracy strata
\cite{Koschorke}.

For integers $p,q\geq 0$, define
\[
A_{p,q}
=
\{T\in \Fred_{p-q}
\mid
\dim\Ker T=p,\ \dim\Coker T=q\}.
\]

\begin{prop}[Koschorke~\cite{Koschorke}]
For each pair $p,q\geq 0$, the subset $A_{p,q}$ is a Banach
submanifold of $\Fred_{p-q}$ of real codimension $2pq$.
\end{prop}

\begin{proof}
See Koschorke~\cite{Koschorke}.
\end{proof}

\subsection{Universal Koschorke classes and characteristic classes}

Let
\[
\nu_{p,q}\to A_{p,q}
\]
be the normal bundle of the Banach submanifold
\(A_{p,q}\subset \Fred_{p-q}\).
By Koschorke's local description of the stratum \(A_{p,q}\) \cite{Koschorke},
the normal bundle is canonically a complex vector bundle of complex rank \(pq\).
Hence it has a Thom class
\[
u_{p,q}\in H^{2pq}(\nu_{p,q},\nu_{p,q}\setminus A_{p,q}).
\]
This Thom class determines a universal cohomology class
\[
k_{p,q}\in H^{2pq}(\Fred_{p-q}),
\]
which we call the \emph{universal Koschorke class}.

Let \(X\) be a paracompact Hausdorff space and let
\[
A\colon X\to \Fred_{p-q}
\]
be a continuous map whose image lies in the fixed index component
\(\Fred_{p-q}\).  Then the universal Koschorke class gives
\[
\chi_{p,q}(A)=A^{*}(k_{p,q})\in H^{2pq}(X).
\]
If \([A]\in [X,\Fred_{p-q}]\) is the homotopy class of \(A\), then we set
\[
\chi_{p,q}([A])=A^*(k_{p,q}).
\]

\begin{dfn}
The class
\[
\chi_{p,q}([A])\in H^{2pq}(X)
\]
is the \emph{\((p,q)\)-th Koschorke class} of the family represented by
\(A\).
\end{dfn}

In particular, when \(p=q\), this construction gives a characteristic class on
the reduced \(K\)-group
\[
\widetilde K(X)=[X,\Fred_0].
\]

\begin{rem}
If \(X\) is not connected and \(A\colon X\to\Fred\) is an arbitrary Fredholm
family, the Fredholm index may vary from one connected component of \(X\) to
another.  Therefore a single finite-dimensional stabilization need not send the
whole of \(X\) into one fixed component \(\Fred_{p-q}\).  One may either work
componentwise or restrict to the fixed component \(\Fred_{p-q}\).  In this paper,
the later use of even Koschorke classes takes place on \(\Fred_0\), so no global assertion
on all of \(K(-)=[-,\Fred]\) is needed.
\end{rem}

\subsection{Interpretation in obstruction theory}

The Koschorke classes satisfy the following basic vanishing property.

\begin{prop}[Koschorke~\cite{Koschorke}]\label{prop:even-vanishing}
Let $X$ be a paracompact Hausdorff space and let
\[
A\colon X\to \Fred_{p-q}
\]
be a continuous map.
Assume that for every \(x\in X\), one has
\[
\dim \Ker A_x<p
\qquad\text{and}\qquad
\dim \Coker A_x<q.
\]
Then
\[
A^{*}(k_{p,q})=0.
\]
Equivalently,
\[
\chi_{p,q}(A)=0.
\]
\end{prop}

\begin{proof}
See Koschorke~\cite{Koschorke}.
\end{proof}

The converse statement, which is valid for CW complexes, is one of the basic
consequences for obstruction theory of the definition of Koschorke classes.

\begin{prop}[Koschorke~\cite{Koschorke}]\label{prop:even-obstruction}
Let $X$ be a CW complex and let
\[
A\colon X\to \Fred_{p-q}
\]
be a continuous map.
If
\[
\chi_{p,q}(A)=A^{*}(k_{p,q})=0,
\]
then there exists a continuous map
\[
B\colon X\to \Fred_{p-q}
\]
such that \(B\) is homotopic to \(A\) and
\[
\dim \Ker B_x<p
\qquad\text{and}\qquad
\dim \Coker B_x<q
\]
for every \(x\in X^{(2pq)}\). Here \(X^{(2pq)}\) denotes the \(2pq\)-skeleton
of \(X\).
\end{prop}

\begin{proof}
See Koschorke~\cite{Koschorke}.
\end{proof}

\subsection{Expression in terms of Chern classes}

We now recall the classical expression of the universal Koschorke classes in
terms of universal Chern classes.

Let
\[
c_i\in H^{2i}(\Fred_{p-q})
\qquad
(i\geq 1)
\]
denote the universal Chern classes.
Then the universal Koschorke class admits a determinantal expression in terms
of these classes.

\begin{thm}[Koschorke~\cite{Koschorke}]\label{thm:even-koschorke-determinant}
For each pair \(p,q\geq 1\), the universal Koschorke class \(k_{p,q}\) is given by
\[
k_{p,q}
=
\det(c_{p-i+j})_{1\leq i,j\leq q}.
\]
\end{thm}

\begin{proof}
See Koschorke~\cite{Koschorke}.
\end{proof}

\begin{example}
The low-degree cases are as follows.

\begin{enumerate}[label=(\arabic*)]
\item For \(p=q=1\), one has
\[
k_{1,1}=\det(c_{1})=c_{1}.
\]

\item For \(p=q=2\), one has
\[
k_{2,2}
=
\det
\begin{pmatrix}
c_{2} & c_{3}\\
c_{1} & c_{2}
\end{pmatrix}
=
c_{2}^{2}-c_{1}c_{3}.
\]
\end{enumerate}
\end{example}

%%%%%%%%%%%%%%%%%%%%%%%%%%%%%%%%%%%%%%%%%%%%%%%%
\section{Odd Koschorke classes}
%%%%%%%%%%%%%%%%%%%%%%%%%%%%%%%%%%%%%%%%%%%%%%%%

\subsection{The space of self-adjoint Fredholm operators and odd \(K\)-theory}

Let $\mathcal{H}$ be a separable infinite-dimensional complex Hilbert space.
We write
\[
\Fred_{\mathrm{sa}}(\mathcal{H})
\]
for the space of bounded self-adjoint Fredholm operators on $\mathcal{H}$,
endowed with the operator norm topology.

Atiyah--Singer proved that $\Fred_{\mathrm{sa}}(\mathcal{H})$ has three
path-connected components: two are contractible, while the remaining
one is a classifying space for odd complex \(K\)-theory \cite{AtiyahSingerSkew}.
We denote this non-contractible component by
\[
\Fred_{\mathrm{sa}}^{1}(\mathcal{H}).
\]

\begin{dfn}
For a paracompact Hausdorff space $X$, we define
\[
K^{1}(X)=[X,\Fred_{\mathrm{sa}}^{1}(\mathcal{H})].
\]
\end{dfn}

The addition on $K^{1}(X)$ is defined in the same way as the addition on
$K(X)$ described in Section~2.1: one uses direct sums and a fixed unitary
isomorphism $\mathcal H\oplus\mathcal H\cong\mathcal H$.

From now on, we suppress the Hilbert space $\mathcal{H}$ in the notation and
write simply
\[
\Fred_{\mathrm{sa}}=\Fred_{\mathrm{sa}}(\mathcal{H}),
\qquad
\Fredsaone=\Fred_{\mathrm{sa}}^{1}(\mathcal{H}).
\]

\subsection{Odd degeneracy strata}

We set
\[
V_{-1}^{\mathrm{sa}}=\varnothing.
\]

\begin{dfn}
For each integer $p\ge 0$, we define
\[
V_p^{\mathrm{sa}}
=
\{A\in \Fred_{\mathrm{sa}}^{1}\mid \dim\ker A\le p\},
\qquad
A_p^{\mathrm{sa}}
=
\{A\in \Fred_{\mathrm{sa}}^{1}\mid \dim\ker A=p\}.
\]
\end{dfn}

\begin{lem}\label{lem:odd-degeneracy-open-closed}
For each integer $p\ge 0$, the following hold.

\smallskip

\noindent
\textup{(1)} The subset $V_p^{\mathrm{sa}}$ is open in $\Fred_{\mathrm{sa}}^{1}$.

\smallskip

\noindent
\textup{(2)} The subset $A_p^{\mathrm{sa}}$ is closed in $V_p^{\mathrm{sa}}$.
\end{lem}

\begin{proof}
The proof is the same as in the even case \cite{Koschorke}.
For \textup{(1)}, suppose that $f_0\in V_p^{\mathrm{sa}}$, and let
\[
r=\dim\ker f_0.
\]
By definition, we have $r\le p$.
Since $f_0$ is a self-adjoint Fredholm operator, we have the orthogonal
decomposition
\[
\mathcal H=\ker f_0\oplus \operatorname{Im} f_0.
\]
With respect to this decomposition, every $f\in \Fred_{\mathrm{sa}}^{1}$ can be written as
\[
f=
\begin{pmatrix}
\alpha & \beta^{*}\\
\beta & \gamma
\end{pmatrix},
\]
where
\[
\alpha\in \operatorname{End}(\ker f_0),\qquad
\beta\in \operatorname{Hom}(\ker f_0,\operatorname{Im} f_0),\qquad
\gamma\in \operatorname{End}(\operatorname{Im} f_0).
\]
We define
\[
\nu_{f_0}
=
\Bigl\{
f\in \Fred_{\mathrm{sa}}^{1}
\;\Big|\;
\gamma:\operatorname{Im} f_0\to \operatorname{Im} f_0
\text{ is invertible}
\Bigr\}.
\]
Since invertibility is an open condition, $\nu_{f_0}$ is an open subset of
$\Fred_{\mathrm{sa}}^{1}$.
Moreover, $f_0\in \nu_{f_0}$, because for $f=f_0$ the lower right block is
\[
\gamma=f_0|_{\operatorname{Im} f_0},
\]
which is invertible on $\operatorname{Im} f_0$.

Now let $f\in \nu_{f_0}$.
For $(x,y)\in \ker f_0\oplus \operatorname{Im} f_0$, the condition
\[
(x,y)\in \ker f
\]
is equivalent to
\[
\alpha x+\beta^{*}y=0,
\qquad
\beta x+\gamma y=0.
\]
Since $\gamma$ is invertible, the second equation is equivalent to
\[
y=-\gamma^{-1}\beta x.
\]
Substituting this into the first equation, we see that $(x,y)\in \ker f$ if and only if
\[
x\in \ker S,
\qquad
y=-\gamma^{-1}\beta x,
\]
where
\[
S=\alpha-\beta^{*}\gamma^{-1}\beta.
\]
Since $f$ is self-adjoint, both $\alpha$ and $\gamma$ are self-adjoint.
Hence $\gamma^{-1}$ is also self-adjoint, and therefore
\[
(\beta^{*}\gamma^{-1}\beta)^{*}
=
\beta^{*}(\gamma^{-1})^{*}\beta
=
\beta^{*}\gamma^{-1}\beta.
\]
Thus $S$ is a Hermitian operator on the finite-dimensional space $\ker f_0$.

The above characterization of $\ker f$ shows that the map
\[
\ker f\longrightarrow \ker S,\qquad (x,y)\longmapsto x
\]
is an isomorphism.
Hence
\[
\dim\ker f=\dim\ker S\le \dim\ker f_0=r\le p.
\]
Therefore $f\in V_p^{\mathrm{sa}}$.
We have proved that
\[
\nu_{f_0}\subset V_p^{\mathrm{sa}}.
\]
Since $\nu_{f_0}$ is an open neighborhood of $f_0$, it follows that
$V_p^{\mathrm{sa}}$ is open in $\Fred_{\mathrm{sa}}^{1}$.
This proves \textup{(1)}.

For \textup{(2)}, the assertion follows immediately from the identity
\[
A_p^{\mathrm{sa}}
=
\{A\in \Fred_{\mathrm{sa}}^{1}\mid \dim\ker A=p\}
=
V_p^{\mathrm{sa}}\setminus V_{p-1}^{\mathrm{sa}}.
\]
Indeed, by \textup{(1)}, the subset $V_p^{\mathrm{sa}}$ is open in
$\Fred_{\mathrm{sa}}^{1}$.
Hence its complement is closed in $\Fred_{\mathrm{sa}}^{1}$, and so
\[
A_p^{\mathrm{sa}}
=
V_p^{\mathrm{sa}}\cap
\bigl(\Fred_{\mathrm{sa}}^{1}\setminus V_{p-1}^{\mathrm{sa}}\bigr)
\]
is closed in $V_p^{\mathrm{sa}}$.
This completes the proof.
\end{proof}

\subsection{The submanifolds $A_p^{\mathrm{sa}}$ and their normal bundles}

\begin{dfn}
Let \(M\) be a Banach manifold.  A subset \(S\subset M\) is called a
\emph{split Banach submanifold} if, for every \(s\in S\), there are a chart
\((U,\varphi)\) of \(M\) at \(s\), Banach spaces \(E_1,E_2\), and an
isomorphism of Banach spaces
\[
\Phi\colon E_1\oplus E_2 \longrightarrow E,
\]
where \(E\) is the model Banach space of the chart, such that, after replacing
\(\varphi\) by \(\Phi^{-1}\circ\varphi\), one has
\[
(\Phi^{-1}\circ\varphi)(U\cap S)
=
(\Phi^{-1}\circ\varphi)(U)\cap(E_1\oplus\{0\}).
\]
See Lang \cite[Chapter~II, \S2]{LangFDG}.
\end{dfn}

\begin{prop}\label{prop:odd-normal-bundle}\label{prop:Asa-submanifold}
For each integer $p\ge 0$, the following hold.

\smallskip

\noindent
\textup{(1)} The subset $A_p^{\mathrm{sa}}$ is a split Banach submanifold of
$\Fred_{\mathrm{sa}}^{1}$ of real codimension $p^2$.  Its normal bundle is
canonically oriented.

\smallskip

\noindent
\textup{(2)} Its normal bundle
\[
N(A_p^{\mathrm{sa}})
=
T\Fred_{\mathrm{sa}}^{1}|_{A_p^{\mathrm{sa}}}/T(A_p^{\mathrm{sa}})
\]
is canonically isomorphic to the vector bundle
\[
\operatorname{Herm}(\ker)\to A_p^{\mathrm{sa}},
\]
where \(\operatorname{Herm}(\ker)\) denotes the vector bundle whose fiber at
$f\in A_p^{\mathrm{sa}}$ is the real vector space
\(\operatorname{Herm}(\ker f)\) of Hermitian endomorphisms of \(\ker f\).
\end{prop}

\begin{proof}
The proof is the same as in the even case \cite{Koschorke}.
We fix $f_0\in A_p^{\mathrm{sa}}$ and set
\[
K=\ker f_0,\qquad W=\operatorname{Im} f_0.
\]
Since $f_0$ is self-adjoint and Fredholm, we have the orthogonal decomposition
\[
\mathcal H=K\oplus W.
\]
Moreover,
\[
\dim_{\mathbb C}K=p.
\]

We first prove \textup{(1)}.  As in the proof of \LaterLem{lem:odd-degeneracy-open-closed}, every $f\in \Fred_{\mathrm{sa}}^{1}$ can
be written with respect to this decomposition in the form
\[
f=
\begin{pmatrix}
\alpha & \beta^{*}\\
\beta & \gamma
\end{pmatrix},
\]
where
\[
\alpha\in \operatorname{Herm}(K),\qquad
\beta\in \operatorname{Hom}(K,W),\qquad
\gamma\in \Fred_{\mathrm{sa}}(W).
\]
Let
\[
\nu_{f_0}
=
\{f\in \Fred_{\mathrm{sa}}^{1}\mid \gamma \text{ is invertible}\}.
\]
This is an open neighborhood of $f_0$.
For $f\in \nu_{f_0}$, define the Schur complement
\[
S(f)=\alpha-\beta^{*}\gamma^{-1}\beta\in \operatorname{Herm}(K).
\]

By the argument in the proof of \LaterLem{lem:odd-degeneracy-open-closed}, for $(x,y)\in K\oplus W$
we have
\[
(x,y)\in \ker f
\iff
x\in \ker S(f),\quad y=-\gamma^{-1}\beta x.
\]
Hence the projection $(x,y)\mapsto x$ induces an isomorphism
\[
\ker f \cong \ker S(f).
\]
Since $\dim_{\mathbb C}K=p$, it follows that for $f\in \nu_{f_0}$,
\[
f\in A_p^{\mathrm{sa}}
\iff
\dim \ker f=p
\iff
\dim \ker S(f)=p
\iff
S(f)=0.
\]
Therefore
\[
A_p^{\mathrm{sa}}\cap \nu_{f_0}=S^{-1}(0).
\]

The map
\[
S\colon \nu_{f_0}\to \operatorname{Herm}(K)
\]
is smooth, because it is obtained from the block entries by addition,
composition, adjoint, and inversion on the open set of invertible operators;
see \cite[Chapter I, \S5, Theorem 5.9]{LangFDG}.

At the point $f_0$, we have
\[
f_0=
\begin{pmatrix}
0 & 0\\
0 & \gamma_0
\end{pmatrix}
\]
with $\gamma_0$ invertible.
Let
\[
\dot f=
\begin{pmatrix}
\dot\alpha & \dot\beta^{*}\\
\dot\beta & \dot\gamma
\end{pmatrix}
\in T_{f_0}\Fred_{\mathrm{sa}}^{1}.
\]
Differentiating the formula
\[
S(f)=\alpha-\beta^{*}\gamma^{-1}\beta
\]
at $f_0$, we obtain
\[
dS_{f_0}(\dot f)=\dot\alpha.
\]
Thus
\[
dS_{f_0}\colon T_{f_0}\Fred_{\mathrm{sa}}^{1}\to \operatorname{Herm}(K)
\]
is surjective.
Its kernel is the closed subspace given by $\dot\alpha=0$, and hence it
splits.
By the Banach-space implicit mapping theorem, or equivalently by the standard
submanifold criterion derived from it
\cite[Chapter II, \S2, Proposition 2.2]{LangFDG}, the zero set
$S^{-1}(0)=A_p^{\mathrm{sa}}\cap \nu_{f_0}$ is a Banach submanifold of
$\nu_{f_0}$.
Since
\[
\dim_{\mathbb R}\operatorname{Herm}(K)=p^2,
\]
this submanifold has real codimension $p^2$.
As $f_0$ was arbitrary, this proves the split Banach-submanifold assertion in
\textup{(1)}.

We prove \textup{(2)}.
For $f\in A_p^{\mathrm{sa}}\cap \nu_{f_0}$, we have $S(f)=0$, and therefore
\[
\Theta_f\colon K\to \ker f,\qquad x\mapsto (x,-\gamma^{-1}\beta x)
\]
is a complex-linear isomorphism.
Since $\Theta_f$ depends smoothly on $f$, these maps give a local
trivialization of the kernel bundle over $A_p^{\mathrm{sa}}\cap \nu_{f_0}$.
Hence $\ker\to A_p^{\mathrm{sa}}$ is a rank-$p$ complex vector bundle.

On the other hand, because $A_p^{\mathrm{sa}}\cap \nu_{f_0}=S^{-1}(0)$ is the
zero set of the submersion $S$, its normal bundle over this neighborhood is
canonically identified with the trivial bundle with fiber $\operatorname{Herm}(K)$:
\[
N(A_p^{\mathrm{sa}})|_{A_p^{\mathrm{sa}}\cap \nu_{f_0}}
\cong
(A_p^{\mathrm{sa}}\cap \nu_{f_0})\times \operatorname{Herm}(K).
\]
Via the isomorphisms $\Theta_f$, we identify $\operatorname{Herm}(K)$ with
$\operatorname{Herm}(\ker f)$.
These local identifications are compatible on overlaps, and therefore glue to a
canonical bundle isomorphism
\[
N(A_p^{\mathrm{sa}})\cong \operatorname{Herm}(\ker).
\]
This proves \textup{(2)}.

Finally, we discuss the orientation.
We choose local unitary frames of the complex vector bundle $\ker$.
Then $\operatorname{Herm}(\ker)$ is locally identified with the fixed real vector
space $\operatorname{Herm}(\mathbb C^p)$.
On overlaps, the transition maps are given by
\[
H\longmapsto uHu^{*},\qquad u\in U(p).
\]
Since $U(p)$ is connected, this action preserves the standard orientation of
$\operatorname{Herm}(\mathbb C^p)$.
Hence $\operatorname{Herm}(\ker)$ is canonically oriented.
By using the isomorphism
\[
N(A_p^{\mathrm{sa}})\cong \operatorname{Herm}(\ker),
\]
we obtain a canonical orientation of the normal bundle.
This completes the proof.
\end{proof}

\subsection{Alexander--Pontrjagin duality and stabilization}

Let \(p\ge 1\) be an integer.  By \LaterProp{prop:odd-normal-bundle},
\(A_p^{\mathrm{sa}}\) is a closed cooriented split Banach submanifold of
\(V_p^{\mathrm{sa}}\) of real codimension \(p^2\).  Moreover,
\[
V_{p-1}^{\mathrm{sa}}=V_p^{\mathrm{sa}}\setminus A_p^{\mathrm{sa}}.
\]

\begin{lem}[Alexander--Pontrjagin duality]\label{lem:odd-eells-duality}
For every \(p\geq1\), there is a canonical isomorphism
\[
\psi_p\colon
H^i(V_p^{\mathrm{sa}},V_{p-1}^{\mathrm{sa}})
\xrightarrow{\ \cong\ }
H^{i-p^2}(A_p^{\mathrm{sa}}).
\]
\end{lem}

\begin{proof}
This is a consequence of the Alexander--Pontrjagin duality theorem of
Eells \cite{EellsAlexanderPontrjagin}.
\end{proof}

\begin{lem}\label{lem:relative-vanishing-vp}
For each integer $p\ge 1$, one has
\[
H^i(V_p^{\mathrm{sa}},V_{p-1}^{\mathrm{sa}})=0
\qquad (i<p^2).
\]
\end{lem}

\begin{proof}
By \LaterLem{lem:odd-eells-duality},
\[
H^i(V_p^{\mathrm{sa}},V_{p-1}^{\mathrm{sa}})
\cong
H^{i-p^2}(A_p^{\mathrm{sa}}).
\]
The group on the right vanishes for \(i<p^2\), because it has negative degree.
This completes the proof.
\end{proof}

\begin{lem}
For integers $r\ge p\ge 1$, the restriction map
\[
H^i(V_{r+1}^{\mathrm{sa}},V_{p-1}^{\mathrm{sa}})
\longrightarrow
H^i(V_r^{\mathrm{sa}},V_{p-1}^{\mathrm{sa}})
\]
is an isomorphism for every $i\le p^2$.
\end{lem}

\begin{proof}
Consider the long exact sequence of the triple
\[
(V_{r+1}^{\mathrm{sa}},V_r^{\mathrm{sa}},V_{p-1}^{\mathrm{sa}}):
\]
\[
\cdots\to
H^i(V_{r+1}^{\mathrm{sa}},V_r^{\mathrm{sa}})
\to
H^i(V_{r+1}^{\mathrm{sa}},V_{p-1}^{\mathrm{sa}})
\to
H^i(V_r^{\mathrm{sa}},V_{p-1}^{\mathrm{sa}})
\to
H^{i+1}(V_{r+1}^{\mathrm{sa}},V_r^{\mathrm{sa}})
\to\cdots .
\]
By \LaterLem{lem:relative-vanishing-vp}, applied with $p$ replaced by $r+1$, we have
\[
H^j(V_{r+1}^{\mathrm{sa}},V_r^{\mathrm{sa}})=0
\qquad (j<(r+1)^2).
\]
Since $r\ge p$, we have $(r+1)^2>p^2$.
Therefore, for every $i\le p^2$, both outer terms in the above exact
sequence vanish, and the middle arrow is an isomorphism.
\end{proof}

\begin{lem}\label{lem:relative-cohomology-fred-vp}
For each integer $p\ge 1$, one has
\[
H^i(\Fred_{\mathrm{sa}}^{1},V_{p-1}^{\mathrm{sa}})
\cong
H^i(V_p^{\mathrm{sa}},V_{p-1}^{\mathrm{sa}})
\qquad (i\le p^2).
\]
\end{lem}

\begin{proof}
By \LaterLem{lem:odd-degeneracy-open-closed}, every \(V_r^{\mathrm{sa}}\) is open in \(\Fred_{\mathrm{sa}}^1\).
Since every self-adjoint Fredholm operator has finite-dimensional kernel,
\[
\Fred_{\mathrm{sa}}^{1}
=
\bigcup_{r\ge p}V_r^{\mathrm{sa}}.
\]
Moreover, every singular simplex in \(\Fred_{\mathrm{sa}}^1\) has compact image.
The function
\[
T\longmapsto \dim_{\C}\ker T
\]
is upper semicontinuous on \(\Fred_{\mathrm{sa}}^1\).  Hence the image of any compact
subset of \(\Fred_{\mathrm{sa}}^1\) is contained in some \(V_N^{\mathrm{sa}}\).
Consequently the relative singular chain complex is the filtered colimit
\[
C_*(\Fred_{\mathrm{sa}}^1,V_{p-1}^{\mathrm{sa}};\Z)
=
\varinjlim_{r\ge p}
C_*(V_r^{\mathrm{sa}},V_{p-1}^{\mathrm{sa}};\Z).
\]
Passing to cochains gives the corresponding inverse system of relative cochain
complexes.  We use Milnor's exact sequence for the cohomology of such an
increasing union; see \cite{Milnor}.  It gives a short exact sequence
\[
0\to
\varprojlim\nolimits^{1}
H^{i-1}(V_r^{\mathrm{sa}},V_{p-1}^{\mathrm{sa}})
\to
H^i(\Fred_{\mathrm{sa}}^{1},V_{p-1}^{\mathrm{sa}})
\to
\varprojlim
H^i(V_r^{\mathrm{sa}},V_{p-1}^{\mathrm{sa}})
\to 0.
\]
By the preceding lemma, for every \(i\le p^2\) the inverse system
\[
\bigl\{H^i(V_r^{\mathrm{sa}},V_{p-1}^{\mathrm{sa}})\bigr\}_{r\ge p}
\]
is constant from \(r=p\) onward.  Hence
\[
\varprojlim\nolimits^{1}
H^{i-1}(V_r^{\mathrm{sa}},V_{p-1}^{\mathrm{sa}})=0
\]
and
\[
\varprojlim
H^i(V_r^{\mathrm{sa}},V_{p-1}^{\mathrm{sa}})
\cong
H^i(V_p^{\mathrm{sa}},V_{p-1}^{\mathrm{sa}}).
\]
Therefore
\[
H^i(\Fred_{\mathrm{sa}}^{1},V_{p-1}^{\mathrm{sa}})
\cong
H^i(V_p^{\mathrm{sa}},V_{p-1}^{\mathrm{sa}})
\qquad (i\le p^2).
\]
This completes the proof.
\end{proof}

\subsection{Universal odd Koschorke classes and characteristic classes}

Let \(p\ge 1\).  Let
\[
1\in H^{0}(A_p^{\mathrm{sa}})
\]
denote the unit class.  By \LaterLem{lem:odd-eells-duality}, the class
\(1\) has a unique inverse image
\[
\psi_p^{-1}(1)
\in
H^{p^2}(V_p^{\mathrm{sa}},V_{p-1}^{\mathrm{sa}}).
\]

\begin{dfn}
We denote this class by
\[
e k_p^{\mathrm{odd}}
:=
\psi_p^{-1}(1)
\in
H^{p^2}(V_p^{\mathrm{sa}},V_{p-1}^{\mathrm{sa}}).
\]
\end{dfn}

By \LaterLem{lem:relative-cohomology-fred-vp}, there is a natural
isomorphism
\[
H^{p^2}(V_p^{\mathrm{sa}},V_{p-1}^{\mathrm{sa}})
\cong
H^{p^2}(\Fred_{\mathrm{sa}}^{1},V_{p-1}^{\mathrm{sa}}).
\]
Accordingly, we also write
\[
e k_p^{\mathrm{odd}}
\in
H^{p^2}(\Fred_{\mathrm{sa}}^{1},V_{p-1}^{\mathrm{sa}})
\]
for the corresponding relative cohomology class.

Let
\[
j_p\colon
(\Fred_{\mathrm{sa}}^{1},\varnothing)
\longrightarrow
(\Fred_{\mathrm{sa}}^{1},V_{p-1}^{\mathrm{sa}})
\]
be the natural inclusion of pairs.
It induces a homomorphism
\[
j_p^{*}\colon
H^{p^2}(\Fred_{\mathrm{sa}}^{1},V_{p-1}^{\mathrm{sa}})
\longrightarrow
H^{p^2}(\Fred_{\mathrm{sa}}^{1},\varnothing)
\cong
H^{p^2}(\Fred_{\mathrm{sa}}^{1}).
\]

\begin{dfn}[\IntroThm{thm:intro-universal-odd-koschorke}]\label{dfn:odd-universal-koschorke}
The class
\[
k_p^{\mathrm{odd}}
=
j_p^{*}(e k_p^{\mathrm{odd}})
\in
H^{p^2}(\Fred_{\mathrm{sa}}^{1})
\]
is called the \emph{$p$-th universal odd Koschorke class}.
\end{dfn}

\begin{dfn}
Let $A$ be a commutative ring.  A \emph{characteristic class of odd $K$-theory
with coefficients in $A$} is a natural transformation
\[
\Phi^{\mathrm{odd}}\colon K^{1}(-)\longrightarrow H^{*}(-;A)
\]
from the functor \(K^{1}(-)\) on paracompact Hausdorff spaces to ordinary
cohomology with coefficients in $A$.
\end{dfn}

\begin{dfn}[\IntroThm{thm:intro-universal-odd-koschorke}]\label{dfn:odd-koschorke-characteristic}
Let $X$ be a paracompact Hausdorff space, and let
\[
[f]\in K^{1}(X)=[X,\Fred_{\mathrm{sa}}^{1}]
\]
be represented by a continuous map
\[
f\colon X\to\Fred_{\mathrm{sa}}^{1}.
\]
The \emph{$p$-th odd Koschorke class} of \([f]\) is defined by
\[
\chi_p^{\mathrm{odd}}([f])=f^{*}(k_p^{\mathrm{odd}})
\in
H^{p^2}(X).
\]
Thus the assignment
\[
\chi_p^{\mathrm{odd}}\colon K^{1}(-)\longrightarrow H^{p^2}(-)
\]
defines a characteristic class of odd $K$-theory.
\end{dfn}

\begin{prop}[\IntroThm{thm:intro-odd-vanishing}]\label{prop:odd-vanishing}
Let $X$ be a paracompact Hausdorff space, and let
\[
A\colon X\to \Fred_{\mathrm{sa}}^{1}
\]
be a continuous map.
Assume that
\[
\dim\ker A_x<p
\]
for every $x\in X$.
Then
\[
\chi_p^{\mathrm{odd}}(A)=0.
\]
\end{prop}

\begin{proof}
By assumption, one has
\[
A(X)\subset V_{p-1}^{\mathrm{sa}}.
\]
Let
\[
i_{p-1}\colon V_{p-1}^{\mathrm{sa}}\hookrightarrow \Fred_{\mathrm{sa}}^{1}
\]
denote the inclusion.
Then $A$ factors through $V_{p-1}^{\mathrm{sa}}$.
More precisely, there exists a continuous map
\[
\widetilde A\colon X\to V_{p-1}^{\mathrm{sa}}
\]
such that
\[
A=i_{p-1}\circ \widetilde A.
\]

Next consider the long exact sequence of the pair
\[
(\Fred_{\mathrm{sa}}^{1},V_{p-1}^{\mathrm{sa}}).
\]
Its relevant part is
\[
H^{p^2}(\Fred_{\mathrm{sa}}^{1},V_{p-1}^{\mathrm{sa}})
\xrightarrow{\,j_p^{*}\,}
H^{p^2}(\Fred_{\mathrm{sa}}^{1})
\xrightarrow{\,i_{p-1}^{*}\,}
H^{p^2}(V_{p-1}^{\mathrm{sa}}).
\]
By exactness, we have
\[
i_{p-1}^{*}\circ j_p^{*}=0.
\]
Since
\[
k_p^{\mathrm{odd}}=j_p^{*}(e k_p^{\mathrm{odd}}),
\]
it follows that
\[
i_{p-1}^{*}(k_p^{\mathrm{odd}})
=
i_{p-1}^{*}j_p^{*}(e k_p^{\mathrm{odd}})
=
0.
\]

By using $A=i_{p-1}\circ \widetilde A$, we obtain
\[
\chi_p^{\mathrm{odd}}(A)
=
A^{*}(k_p^{\mathrm{odd}})
=
\widetilde A^{*}i_{p-1}^{*}(k_p^{\mathrm{odd}})
=
0.
\]
This completes the proof.
\end{proof}

We next recall the cohomology ring of the space of self-adjoint Fredholm operators.
From the standard cohomology calculation for the stable unitary group, one has
an isomorphism
of graded rings
\[
H^{*}(\Fredsaone)
\cong
\Lambda_{\Z}
\bigl(c_{1/2},c_{3/2},c_{5/2},\ldots\bigr).
\]
Here
\[
c_{(2i+1)/2}\in H^{2i+1}(\Fredsaone)
\qquad (i\geq 0)
\]
denotes the universal odd Chern class.  More precisely, under the comparison
of \(\Fredsaone\) with the stable unitary model, \(c_{(2i+1)/2}\) corresponds
to the transgression of the ordinary Chern class \(c_{i+1}\); see
\cite{Borel,Cibotaru,QuillenCayley}.

\begin{thm}[\IntroThm{thm:intro-odd-chern-formula}]\label{thm:odd-koschorke-chern}
For every integer \(p\geq 1\), the \(p\)-th universal odd Koschorke class is
given by
\[
k_p^{\mathrm{odd}}
=
c_{1/2}c_{3/2}\cdots c_{(2p-1)/2}
\in
H^{p^2}(\Fredsaone).
\]
\end{thm}

\begin{proof}
The proof is given in \LaterCor{cor:odd-koschorke-product-paper} below.
The essential input is Cibotaru's construction of geometric representatives
for the integral cohomology generators of a classifying space for odd
\(K\)-theory \cite{Cibotaru}.  In that corollary we apply this construction to
the initial Schubert condition \(J=(0,1,\ldots,p-1)\), identify the resulting
locus with the kernel-dimension locus \(\{\dim\ker A\ge p\}\), and verify that
Cibotaru's coorientation agrees with the coorientation used in the duality
definition of \(k_p^{\mathrm{odd}}\).
\end{proof}

%%%%%%%%%%%%%%%%%%%%%%%%%%%%%%%%%%%%%%%%%%%%%%%%
\subsection{Interpretation in obstruction theory}
%%%%%%%%%%%%%%%%%%%%%%%%%%%%%%%%%%%%%%%%%%%%%%%%

We now describe the meaning as obstruction classes of the odd Koschorke
classes.

Let \(p\geq 2\), and put
\[
F=\Fredsaone,\qquad
V=V_{p-1}^{\mathrm{sa}},\qquad
n=p^{2}.
\]
Thus
\[
V=\{T\in \Fredsaone\mid \dim_{\C}\ker T<p\}.
\]

\begin{prop}\label{prop:odd-pair-connected}
For \(p\geq 2\), the pair
\[
(F,V)=(\Fredsaone,V_{p-1}^{\mathrm{sa}})
\]
is \((n-1)\)-connected.
More precisely,
\[
\pi_i(F,V)=0
\qquad
(1\leq i<n),
\]
and the map
\[
\pi_0(V)\longrightarrow \pi_0(F)
\]
is bijective.
\end{prop}

\begin{proof}
We recall only the part of the argument that will be used below.
For \(r\geq p\), one has
\[
V_{r-1}^{\mathrm{sa}}
=
V_r^{\mathrm{sa}}\setminus A_r^{\mathrm{sa}}.
\]
By \LaterProp{prop:Asa-submanifold}, the subset
\(A_r^{\mathrm{sa}}\) is a closed split Banach submanifold of
\(V_r^{\mathrm{sa}}\) of real codimension \(r^2\).

Let \(P\) be a finite-dimensional compact smooth manifold, and let
\(C\subset P\) be a closed subset.  Suppose that a continuous map
\[
f\colon P\longrightarrow V_r^{\mathrm{sa}}
\]
carries a neighborhood of \(C\) into \(V_{r-1}^{\mathrm{sa}}\).  First use
relative smooth approximation for maps from finite-dimensional manifolds to
open subsets of Banach spaces, fixing a smaller closed neighborhood of \(C\),
to replace \(f\) by a smooth map in the same relative homotopy class.  We then
apply the relative Thom--Smale transversality theorem for maps from a
finite-dimensional manifold to a Banach manifold with respect to the
split submanifold of finite codimension \(A_r^{\mathrm{sa}}\)
\cite{AbrahamRobbin,Palais,Smale}.  Thus we may deform \(f\), fixing a closed
neighborhood of \(C\), so that it is transverse to \(A_r^{\mathrm{sa}}\) outside
that neighborhood.  If
\[
\dim P<r^2,
\]
then the inverse image of \(A_r^{\mathrm{sa}}\) is a smooth manifold of
negative dimension.  Hence it is empty.  It follows that the deformed map
has image in \(V_{r-1}^{\mathrm{sa}}\).

Applying this argument to disks and to homotopies of spheres gives
\[
\pi_i(V_r^{\mathrm{sa}},V_{r-1}^{\mathrm{sa}})=0
\qquad
(1\leq i<r^2).
\]
The same argument applied to paths shows that
\[
\pi_0(V_{r-1}^{\mathrm{sa}})
\longrightarrow
\pi_0(V_r^{\mathrm{sa}})
\]
is bijective for \(r\geq 2\).

Now use the long exact sequence of relative homotopy groups for the triple
\[
V_{p-1}^{\mathrm{sa}}
\subset
V_{r-1}^{\mathrm{sa}}
\subset
V_r^{\mathrm{sa}}.
\]
Since \(r\geq p\), the preceding vanishing implies, by induction on \(r\),
that
\[
\pi_i(V_N^{\mathrm{sa}},V_{p-1}^{\mathrm{sa}})=0
\qquad
(1\leq i<p^2)
\]
for every \(N\geq p\).  The same induction gives that
\[
\pi_0(V_{p-1}^{\mathrm{sa}})
\longrightarrow
\pi_0(V_N^{\mathrm{sa}})
\]
is bijective.

Finally, every compact subset of \(\Fredsaone\) is contained in
\(V_N^{\mathrm{sa}}\) for some \(N\).  Indeed, the kernel dimension is upper
semicontinuous for Fredholm operators.  Therefore every map from a compact
disk, sphere, or interval into \(\Fredsaone\) factors through some
\(V_N^{\mathrm{sa}}\).  Passing from the finite stages to \(\Fredsaone\), we
obtain
\[
\pi_i(\Fredsaone,V_{p-1}^{\mathrm{sa}})=0
\qquad
(1\leq i<p^2),
\]
and
\[
\pi_0(V_{p-1}^{\mathrm{sa}})
\longrightarrow
\pi_0(\Fredsaone)
\]
is bijective.  This completes the proof.
\end{proof}

\begin{lem}\label{lem:odd-obstruction-local-system}
Let \(p\geq2\), put \(F=\Fredsaone\), \(V=V_{p-1}^{\mathrm{sa}}\), and
\(n=p^2\).  The local coefficient system on \(V\) with fiber
\(\pi_n(F,V)\) is canonically identified with the constant system \(\Z\).
Under this identification, the relative class
\[
\widetilde{k}^{\mathrm{odd}}_p\in H^n(F,V)
\]
corresponds to the generator of the obstruction coefficient group.
\end{lem}

\begin{proof}
By \LaterProp{prop:odd-pair-connected}, the pair \((F,V)\) is
\((n-1)\)-connected.  Hence the relative Hurewicz theorem identifies the first
nonzero relative homotopy group with relative homology; see
\cite[Chapter~IV, \S7]{Whitehead}:
\[
h_v\colon \pi_n(F,V;v)\xrightarrow{\ \cong\ }H_n(F,V;\Z)
\]
for every base point \(v\in V\).  The obstruction coefficient system is the
local system whose fiber at \(v\) is \(\pi_n(F,V;v)\), with monodromy given by
change of base point along paths in \(V\).

We now identify this system with the local orientation system of the first
stratum.  By the local model given by the Schur complement in
\LaterProp{prop:odd-normal-bundle}, if \(U\) is a sufficiently small
neighborhood of \(A\) in \(F\), then the pair \((U,U\cap V)\) is identified
with the product of a tangential factor and the pair in the normal direction
\[
\bigl(\operatorname{Herm}(\ker A),
\operatorname{Herm}(\ker A)\setminus\{0\}\bigr).
\]
Therefore the local relative homology group is canonically identified with
\[
H_n\bigl(\operatorname{Herm}(\ker A),
\operatorname{Herm}(\ker A)\setminus\{0\};\Z\bigr)\cong\Z.
\]
The chosen generator is the fundamental class of a positively oriented normal
\(n\)-disk.  The relative class \(\widetilde{k}^{\mathrm{odd}}_p\) is characterized by
the property that it evaluates to \(1\) on this generator.

It remains to check that this generator has trivial monodromy.  If a path in
\(A_p^{\mathrm{sa}}\) changes a unitary frame of the kernel by \(u\in U(p)\),
then the induced map on the normal fiber is
\[
H\longmapsto uHu^*\qquad (H\in \operatorname{Herm}(\C^p)).
\]
The group \(U(p)\) is connected, and this conjugation action preserves the
standard real orientation of \(\operatorname{Herm}(\C^p)\).  Thus the local system of orientations is trivial.  Equivalently,
a positively oriented normal disk remains positively oriented after transport
along any loop.  Since
the relative Hurewicz isomorphism is natural with respect to change of
base point, the same trivial monodromy holds for the system
\(v\mapsto\pi_n(F,V;v)\).

Therefore the obstruction coefficient system is canonically the constant
system \(\Z\).  Under the identification given by the Hurewicz theorem, the homomorphism
\[
\widehat{k}^{\mathrm{odd}}_p\colon \pi_n(F,V;v)\longrightarrow \Z,
\qquad
\alpha\longmapsto
\bigl\langle\widetilde{k}^{\mathrm{odd}}_p,h_v(\alpha)\bigr\rangle
\]
sends the positive normal generator to \(1\), and hence identifies the relative class with the generator of the obstruction coefficient group.
\end{proof}

Let \(X\) be a CW complex, and let
\[
A\colon X\longrightarrow F
\]
be a continuous map.  Since the pair \((F,V)\) is \((n-1)\)-connected, the
standard obstruction theory allows us to deform \(A\) so that
\[
A(X^{(n-1)})\subset V.
\]
Here \(X^{(n-1)}\) denotes the \((n-1)\)-skeleton of \(X\).  After this
preliminary deformation, each oriented \(n\)-cell \(e\) of \(X\), with
characteristic map
\[
\varphi_e\colon (D^n,S^{n-1})\longrightarrow (X^{(n)},X^{(n-1)}),
\]
gives a relative homotopy class
\[
[A\circ\varphi_e]\in \pi_n(F,V).
\]
We use cellular cochains in the following precise sense.  The cellular chain
and cochain groups are
\[
C_n^{\mathrm{cell}}(X;\Z)=H_n(X^{(n)},X^{(n-1)};\Z),
\qquad
C^n_{\mathrm{cell}}(X;\Z)=\operatorname{Hom}(C_n^{\mathrm{cell}}(X;\Z),\Z).
\]
After orientations of the \(n\)-cells have been chosen, an element of
\(C^n_{\mathrm{cell}}(X;\Z)\) is determined by its values on the oriented
\(n\)-cells.  We define
\[
o_p^{\mathrm{odd}}(A)\in C^n_{\mathrm{cell}}(X;\Z)
\]
by
\[
o_p^{\mathrm{odd}}(A)(e)
=
\widehat{k}^{\mathrm{odd}}_p
\bigl([A\circ\varphi_e]\bigr).
\]

\begin{prop}\label{prop:odd-obstruction-cocycle}
The cellular cochain \(o_p^{\mathrm{odd}}(A)\) is a cocycle.  Its cellular
cohomology class is independent of the preliminary deformation of \(A\) on
\(X^{(n-1)}\).  Under the standard isomorphism between cellular cohomology and singular
cohomology for CW complexes, we denote the corresponding singular cohomology
class by
\[
O_p^{\mathrm{odd}}(A)
=
[o_p^{\mathrm{odd}}(A)]
\in
H^n(X).
\]
\end{prop}

\begin{proof}
This is the standard primary obstruction construction for the problem of
deforming a map into a subspace; see, for example, \cite{Whitehead}.
\LaterProp{prop:odd-pair-connected} shows that there are no lower
obstructions.  The obstruction on an \(n\)-cell is measured by the relative
homotopy class of
\[
A\circ\varphi_e\colon (D^n,S^{n-1})\longrightarrow (F,V),
\]
and the coefficient group is identified with \(\Z\) by
\(\widehat{k}^{\mathrm{odd}}_p\).

The cocycle condition follows from the fact that the sum of the obstruction
classes on the boundary of an \((n+1)\)-cell is represented by a map that
extends over that cell.  If one chooses another preliminary deformation on
\(X^{(n-1)}\), the resulting obstruction cochain changes by a cellular
coboundary.  Hence the cohomology class is independent of the choice.
This completes the proof.
\end{proof}

\begin{thm}[\IntroThm{thm:intro-odd-obstruction}]\label{thm:odd-koschorke-obstruction}
Let \(p\geq 2\), let \(X\) be a CW complex, and let
\[
A\colon X\longrightarrow \Fredsaone
\]
be a continuous map.  Then
\[
O_p^{\mathrm{odd}}(A)
=
\chi_p^{\mathrm{odd}}(A)
\in
H^{p^2}(X).
\]
Thus the \(p\)-th odd Koschorke class is the primary obstruction to
deforming \(A\) into \(V_{p-1}^{\mathrm{sa}}\).
\end{thm}

\begin{proof}
Let \(n=p^2\).  By \LaterProp{prop:odd-pair-connected}, we may deform
\(A\) so that
\[
A(X^{(n-1)})\subset V.
\]
Since \(\chi_p^{\mathrm{odd}}(A)\) is homotopy invariant, this preliminary
deformation does not change the class to be computed.

The map \(A\) now defines a map of pairs
\[
(X^{(n)},X^{(n-1)})
\longrightarrow
(F,V).
\]
Thus the relative class gives
\[
A^*\widetilde{k}^{\mathrm{odd}}_p
\in
H^n(X^{(n)},X^{(n-1)}).
\]
By excision and the quotient description of relative cohomology,
\[
H^n(X^{(n)},X^{(n-1)})
\cong
\widetilde H^n(X^{(n)}/X^{(n-1)}).
\]
Since \(X^{(n)}/X^{(n-1)}\) is a wedge of \(n\)-spheres indexed by the
\(n\)-cells of \(X\), the chosen orientations of the \(n\)-cells identify this
relative group with the cellular cochain group
\[
C^n_{\mathrm{cell}}(X;\Z)
=
\operatorname{Hom}(H_n(X^{(n)},X^{(n-1)};\Z),\Z).
\]
Under this identification, the value of the associated cellular cochain on an
oriented \(n\)-cell \(e\) is
\[
\begin{aligned}
\bigl\langle
A^*\widetilde{k}^{\mathrm{odd}}_p,\,[e]
\bigr\rangle
&=
\bigl\langle
(A\circ\varphi_e)^*\widetilde{k}^{\mathrm{odd}}_p,\,
[D^n,S^{n-1}]
\bigr\rangle \\
&=
\bigl\langle
\widetilde{k}^{\mathrm{odd}}_p,\,
h([A\circ\varphi_e])
\bigr\rangle \\
&=
o_p^{\mathrm{odd}}(A)(e).
\end{aligned}
\]
Therefore \(A^*\widetilde{k}^{\mathrm{odd}}_p\) represents the obstruction
cocycle \(o_p^{\mathrm{odd}}(A)\).

On the other hand, the odd Koschorke class is defined by
\[
k_p^{\mathrm{odd}}
=
j^*\widetilde{k}^{\mathrm{odd}}_p
\in
H^n(F),
\]
where
\[
j\colon (F,\emptyset)\longrightarrow (F,V)
\]
is the natural map of pairs.  Hence the cellular cohomology class represented by the associated cellular
cocycle satisfies
\[
[A^*\widetilde{k}^{\mathrm{odd}}_p]_{\mathrm{cell}}
=
A^*j^*\widetilde{k}^{\mathrm{odd}}_p
=
A^*k_p^{\mathrm{odd}}
=
\chi_p^{\mathrm{odd}}(A),
\]
where the last equality is interpreted through the isomorphism between cellular
cohomology and singular cohomology for CW complexes.
It follows that
\[
O_p^{\mathrm{odd}}(A)=\chi_p^{\mathrm{odd}}(A).
\]
This completes the proof.
\end{proof}

\begin{cor}\label{cor:odd-koschorke-avoidance}
Let \(p\geq 2\), let \(X\) be a CW complex, and let
\[
A\colon X\longrightarrow \Fredsaone
\]
be a continuous map.  If
\[
\chi_p^{\mathrm{odd}}(A)=0
\in
H^{p^2}(X),
\]
then there is a map
\[
A'\colon X^{(p^2)}\longrightarrow \Fredsaone
\]
such that \(A'\) is homotopic to \(A|_{X^{(p^2)}}\) and
\[
\dim_{\C}\ker A'(x)<p
\qquad
(x\in X^{(p^2)}).
\]
\end{cor}

\begin{proof}
By \LaterThm{thm:odd-koschorke-obstruction},
\[
O_p^{\mathrm{odd}}(A)=\chi_p^{\mathrm{odd}}(A).
\]
Hence the assumption implies that the primary obstruction to deforming
\(A\) into \(V_{p-1}^{\mathrm{sa}}\) is zero.  Since \((F,V)\) is
\((p^2-1)\)-connected, there are no lower obstructions.  The standard
extension step in obstruction theory therefore deforms
\(A|_{X^{(p^2)}}\) into \(V\).  This gives the required map \(A'\).
\end{proof}

\begin{rem}\label{rem:p-equals-one-spectral-flow}
The case \(p=1\) is not covered by the preceding argument.  In that case the
first obstruction is one-dimensional and is governed by spectral flow
\cite{AtiyahPatodiSingerIII,AtiyahSingerSkew,Phillips}.  Thus the case \(p=1\) should be treated
separately as the degree-one obstruction to removing zero eigenvalues.
\end{rem}

%%%%%%%%%%%%%%%%%%%%%%%%%%%%%%%%%%%%%%%%%%%%%%%%
\subsection{Twisted characteristic classes}
%%%%%%%%%%%%%%%%%%%%%%%%%%%%%%%%%%%%%%%%%%%%%%%%

We next explain why the real images of the odd Koschorke classes satisfy the
condition which is needed for characteristic classes in odd twisted
\(K\)-theory.  In this subsection, when an integral cohomology class is used
with real coefficients, we use the same symbol for its image under
\[
H^*(-)\longrightarrow H^*(-;\R).
\]

Atiyah and Segal studied universal characteristic classes for twisted
\(K\)-theory in ordinary cohomology \cite{AtiyahSegalTwisted,AtiyahSegal}.  In
the even case, their condition can be expressed as membership in the kernel of
the derivation on the polynomial algebra
\[
\mathbb{Q}[x_1,x_2,x_3,\ldots],
\]
where \(x_i\) denotes the degree-\(2i\) component of the universal Chern
character, normalized by \(p_i=i!x_i\). It is defined by
\[
D_{\mathrm{ev}}(x_1)=0,
\qquad
D_{\mathrm{ev}}(x_i)=x_{i-1}\quad (i\geq2).
\]  Thus a universal even class whose polynomial
representative lies in \(\ker D_{\mathrm{ev}}\) gives a characteristic class
with real coefficients for twisted \(K\)-theory.

For odd twisted \(K\)-theory, Gomi studied the corresponding invariant
condition in his work on Mickelsson's twisted \(K\)-theory invariant
\cite{GomiMickelsson}.  We use the self-adjoint Fredholm model
\(\Fred_{\mathrm{sa}}^1\).  Its real cohomology algebra is
\[
\mathcal A_{\mathrm{odd}}
=
\Lambda_{\R}(x_1,x_3,x_5,\ldots),
\qquad
\deg x_{2i+1}=2i+1,
\]
where \(x_{2i+1}\) is the real image of the universal odd Chern class
\(c_{(2i+1)/2}\in H^{2i+1}(\Fred_{\mathrm{sa}}^1)\).  The relevant condition
is membership in the kernel of the derivation
\[
D_{\mathrm{odd}}\colon
\mathcal A_{\mathrm{odd}}
\longrightarrow
\mathcal A_{\mathrm{odd}}
\]
defined by
\[
D_{\mathrm{odd}}(x_1)=0,
\qquad
D_{\mathrm{odd}}(x_{2i+1})=x_{2i-1}\quad (i\geq1).
\]
The theorem below shows that the odd Koschorke classes satisfy this condition.

\begin{thm}[Odd Koschorke classes for twisted \(K\)-theory]\label{thm:odd-koschorke-twisted-characteristic}
For every \(p\geq 1\), the real image of the \(p\)-th universal odd Koschorke
class satisfies
\[
D_{\mathrm{odd}}k_p^{\mathrm{odd}}=0.
\]
Equivalently,
\[
k_p^{\mathrm{odd}}\in \ker D_{\mathrm{odd}}\subset H^*(\Fred_{\mathrm{sa}}^1;\R).
\]
\end{thm}

\begin{proof}
We write \(x_{2i+1}\) for the real generator corresponding to
\(c_{(2i+1)/2}\).  By \LaterThm{thm:odd-koschorke-chern},
\[
k_p^{\mathrm{odd}}=x_1x_3\cdots x_{2p-1}.
\]
Since \(D_{\mathrm{odd}}(x_1)=0\) and
\(D_{\mathrm{odd}}(x_{2i+1})=x_{2i-1}\), the derivation rule gives a sum in
which every nonzero summand contains some odd generator twice.  The cohomology
ring is an exterior algebra, so each such summand vanishes.  Hence
\(D_{\mathrm{odd}}k_p^{\mathrm{odd}}=0\).
\end{proof}
%%%%%%%%%%%%%%%%%%%%%%%%%%%%%%%%%%%%%%%%%%%%%%%%

%%%%%%%%%%%%%%%%%%%%%%%%%%%%%%%%%%%%%%%%%%%%%%%%

%%%%%%%%%%%%%%%%%%%%%%%%%%%%%%%%%%%%%%%%%%%%%%%%
\section{Generalized Koschorke classes}
%%%%%%%%%%%%%%%%%%%%%%%%%%%%%%%%%%%%%%%%%%%%%%%%

The discussion of the Neveu--Schwarz algebra in Section~5 requires a basis of
\[
H^*(\Fredzero;\C)
\]
formed by even generalized Koschorke classes.
The special kernel-dimension classes on \(\Fredzero\) are not sufficient for
this purpose.  In this section we therefore construct, for every partition
\(\lambda\), a generalized Koschorke class
\[
[\Sigma_\lambda]\in H^{2|\lambda|}(\Fredzero).
\]
The geometric construction is most naturally made on the component \(\Fred_k\),
where the generic kernel has dimension \(k\).  After proving the Giambelli
formula on \(\Fred_k\), we transport the class to \(\Fredzero\) by a standard
shift equivalence.  We then define the generalized odd Koschorke classes needed
for the proof of \LaterThm{thm:odd-koschorke-chern}.

%%%%%%%%%%%%%%%%%%%%%%%%%%%%%%%%%%%%%%%%%%%%%%%%

\subsection{Cooriented quasi-manifolds}

We use Cibotaru's notion of a cooriented quasi-manifold \cite[Definition A.5]{Cibotaru}.  In
this paper we only use the definition of a quasi-manifold and the associated
extension of the Thom class.

\begin{dfn}\label{dfn:cooriented-quasi-manifold-paper}
Let \(X\) be a Banach manifold and let \(c\ge0\).  A closed subset
\[
W\subset X
\]
is called a \emph{quasi-manifold of codimension \(c\)} if it is equipped with
a decreasing filtration by closed subsets
\[
W=F_0\supset F_1\supset F_2\supset F_3\supset\cdots
\]
satisfying the following two conditions.

\begin{enumerate}[label=\textup{(\arabic*)}]
\item \(F_1=F_2\).
\item For every \(m\ge0\), the stratum \(S_m=F_m\setminus F_{m+1}\) is either
empty or a Banach submanifold of \(X\) of codimension \(c+m\).
\end{enumerate}

The stratum \(S_0=F_0\setminus F_1\) is called the \emph{top stratum}.  If the
normal bundle of \(S_0\subset X\) is oriented, then \(W\) is called
\emph{cooriented}.  A coorientation means an orientation of this normal bundle.
\end{dfn}

\begin{lem}\label{lem:top-stratum-closed-paper}
In \LaterDfn{dfn:cooriented-quasi-manifold-paper}, the top stratum
\[
S_0=F_0\setminus F_1=F_0\setminus F_2
\]
is a closed Banach submanifold of \(X\setminus F_2\).
\end{lem}

\begin{proof}
Since \(F_1=F_2\), one has
\[
S_0=F_0\setminus F_1=F_0\setminus F_2=(X\setminus F_2)\cap F_0.
\]
Because \(F_0\) is closed in \(X\), the subset \(S_0\) is closed in
\(X\setminus F_2\).  Its submanifold property is part of the definition of the
filtration.
\end{proof}

We shall use the following results of Cibotaru on cooriented quasi-manifolds
\cite{Cibotaru}.

\begin{thm}[Cibotaru~{\cite[Proposition A.5]{Cibotaru}}]\label{thm:cibotaru-extension-detailed}
Let \(W\subset X\) be a cooriented quasi-manifold of codimension \(c\), with
filtration
\[
W=F_0\supset F_1=F_2\supset F_3\supset\cdots .
\]
Let \(S_0=F_0\setminus F_1\) be the top stratum and let \(\omega\) be the chosen
coorientation.  Then the Thom class of the closed cooriented submanifold
\[
S_0\subset X\setminus F_2
\]
extends uniquely to a cohomology class
\[
[W,\omega]_X\in H^c(X).
\]
Equivalently, the restriction of \([W,\omega]_X\) to \(X\setminus F_2\) is the
Thom class of \(S_0\subset X\setminus F_2\).
\end{thm}

\begin{dfn}\label{dfn:qm-transversality-paper}
Let \(W\subset X\) be a quasi-manifold with a filtration
\[
W=F_0\supset F_1=F_2\supset F_3\supset\cdots,
\]
and let \(S_m=F_m\setminus F_{m+1}\) denote its strata.  A smooth map
\[
f\colon M\longrightarrow X
\]
from a finite-dimensional smooth manifold \(M\) is said to be \emph{transverse to
\(W\)} if it is transverse, in the usual Banach-manifold sense, to every stratum
\(S_m\) for which \(S_m\) is non-empty.
\end{dfn}

\begin{thm}[Cibotaru~{\cite[Proposition A.5]{Cibotaru}}]\label{thm:cibotaru-naturality-detailed}
Let \(W\subset X\) be a cooriented quasi-manifold of codimension \(c\), and
let
\[
f\colon M\longrightarrow X
\]
be a smooth map from a finite-dimensional smooth manifold.  Suppose that \(f\)
is transverse to \(W\) in the sense of \LaterDfn{dfn:qm-transversality-paper}.  Then
\(f^{-1}(W)\) is a cooriented quasi-manifold of \(M\) of codimension \(c\), and
\[
f^*[W,\omega]_X=[f^{-1}(W),f^*\omega]_M\in H^c(M).
\]
\end{thm}

\begin{thm}[cf. Cibotaru~{\cite[Proposition A.5]{Cibotaru}}]\label{thm:cibotaru-open-diffeomorphism-detailed}
Let \(U\subset X\) be an open subset and let \(W\subset X\) be a cooriented
quasi-manifold of codimension \(c\).  We set \(W_U=W\cap U\).  Then
\(W_U\subset U\) is a cooriented quasi-manifold of codimension \(c\), and
\[
[W_U]_U=\iota^*[W]_X,
\]
where \(\iota\colon U\hookrightarrow X\) is the inclusion.  Moreover, if
\(\varphi\colon Y\xrightarrow{\cong}U\) is a diffeomorphism of Banach manifolds,
then \(\varphi^{-1}(W_U)\subset Y\) is a cooriented quasi-manifold of
codimension \(c\), and
\[
[\varphi^{-1}(W_U)]_Y=\varphi^*[W_U]_U.
\]
\end{thm}

\begin{lem}\label{lem:comparison-two-quasi-paper}
Let \(W\subset X\) be a closed subset equipped with two cooriented
quasi-manifold structures of the same codimension \(c\).  Suppose that the
two classes are both characterized as extensions of the same Thom class
of the same cooriented top stratum under the extension theorem.  Then the two
associated classes in \(H^c(X)\) are equal.
\end{lem}

\subsection{Quasi-manifolds on \(\Fred_k\)}
%%%%%%%%%%%%%%%%%%%%%%%%%%%%%%%%%%%%%%%%%%%%%%%%

We fix a complete decreasing flag of closed complex subspaces
\[
\mathcal H=W_0\supset W_1\supset W_2\supset\cdots,
\qquad
\operatorname{codim}_{\C}W_j=j,
\qquad
\bigcap_{j\ge0}W_j=\{0\}.
\]
For example, after choosing an orthonormal basis \((e_0,e_1,e_2,\ldots)\), one
may take \(W_j\) to be the closure of the span of
\(e_j,e_{j+1},e_{j+2},\ldots\).

\begin{lem}\label{lem:finite-subspace-avoids-WN-paper}
Let \(V\subset\mathcal H\) be a finite-dimensional subspace.  Then there exists
\(N\) such that
\[
V\cap W_N=0.
\]
\end{lem}

\begin{proof}
Assume that no such integer \(N\) exists.  Then, for every \(N\), there is a
unit vector \(v_N\in V\cap W_N\).  The unit sphere of the finite-dimensional
space \(V\) is compact; hence some subsequence \((v_{N_\ell})_\ell\) converges
to a unit vector \(v\in V\).  Let \(m\ge0\) be arbitrary.  Since
\(N_\ell\to\infty\), all sufficiently large terms of the subsequence lie in
\(W_m\).  The subspace \(W_m\) is closed, so the limit vector \(v\) also lies in
\(W_m\).  As this is true for every \(m\), one obtains
\[
v\in\bigcap_{m\ge0}W_m.
\]
By the defining property of the flag, this intersection is \(\{0\}\).  This contradicts
the fact that \(v\) is a unit vector.  Therefore \(V\cap W_N=0\) for some \(N\).
\end{proof}

Let \(k\ge0\), and let
\[
\lambda=(\lambda_1\ge\lambda_2\ge\cdots\ge\lambda_k\ge0)
\]
be a partition of length at most \(k\), with trailing zeros added if necessary.
We set
\[
|\lambda|=\lambda_1+\lambda_2+\cdots+\lambda_k.
\]

\begin{dfn}\label{dfn:fredk-degeneracy-locus-paper}
The \emph{degeneracy locus of type \(\lambda\) in \(\Fred_k\)} is
\[
\Sigma^{(k)}_\lambda
=
\left\{
T\in\Fred_k
\ \middle|\
\dim_\C(\Ker T\cap W_{\lambda_i+k-i})\ge i
\quad (i=1,\ldots,k)
\right\}.
\]
For \(k=0\) and \(\lambda=\varnothing\), we set
\(\Sigma^{(0)}_\varnothing=\Fredzero\).
\end{dfn}

\begin{dfn}\label{dfn:exact-type-paper}
Let \(r\ge0\), and let
\[
J=(j_0<j_1<\cdots<j_{r-1})
\]
be a strictly increasing sequence of non-negative integers.  For an
\(r\)-dimensional subspace \(V\subset\mathcal H\), we say that \(V\) has
\emph{exact type \(J\)} if
\[
\dim(V\cap W_{j_q})=r-q,
\qquad
\dim(V\cap W_{j_q+1})=r-q-1
\qquad(q=0,\ldots,r-1).
\]
For \(T\in\Fred_k\), we say that \(T\) has exact type \(J\) if \(\Ker T\) has
exact type \(J\).  We denote the corresponding exact stratum by
\[
\Sigma^{(k),\circ}_{J}
=
\{T\in\Fred_k\mid \Ker T\text{ has exact type }J\}.
\]
\end{dfn}

\begin{lem}\label{lem:exact-type-exists-paper}
Every finite-dimensional subspace \(V\subset\mathcal H\) has a unique exact
type.
\end{lem}

\begin{proof}
We set \(r=\dim V\).  By \LaterLem{lem:finite-subspace-avoids-WN-paper}, choose
\(N\) such that \(V\cap W_N=0\).  Let
\[
d_m=\dim(V\cap W_m)
\qquad(m\ge0).
\]
Then \(d_0=r\), and \(d_m=0\) for all sufficiently large \(m\).  Since
\(W_m/W_{m+1}\) is one-dimensional, one has
\[
0\le d_m-d_{m+1}\le1.
\]
Thus the function \(m\mapsto d_m\) drops by one exactly \(r\) times.  We define
\(j_q\) to be the unique integer at which \(d_m\) drops from \(r-q\) to
\(r-q-1\).  Then
\[
0\le j_0<j_1<\cdots<j_{r-1},
\]
and the defining equalities of exact type hold.  Uniqueness follows because the
positions of the jumps of the sequence \((d_m)\) are uniquely determined by
\(V\).
\end{proof}

\begin{dfn}\label{dfn:k-domination-paper}
Let \(J=(j_0<\cdots<j_{r-1})\) be of length \(r\ge k\).  We say that \(J\)
\emph{\(k\)-dominates \(\lambda\)}, and write
\[
J\succeq_k\lambda,
\]
if
\[
j_{r-i}\ge \lambda_i+k-i
\qquad(i=1,\ldots,k).
\]
We also set
\[
J^{(k)}_\lambda=(\lambda_k,\lambda_{k-1}+1,\ldots,\lambda_1+k-1).
\]
This is a strictly increasing sequence of length \(k\).
\end{dfn}

\begin{lem}\label{lem:schubert-combinatorics-paper}
The locus \(\Sigma^{(k)}_\lambda\) has the disjoint decomposition
\[
\Sigma^{(k)}_\lambda
=
\bigsqcup_{J\succeq_k\lambda}\Sigma^{(k),\circ}_{J}.
\]
\end{lem}

\begin{proof}
Let \(V=\Ker T\).  If \(V\) has exact type \(J\), then for every \(m\ge0\)
\[
\dim(V\cap W_m)=\#\{q\mid j_q\ge m\}.
\]
Indeed, the exact type records exactly the jump positions of the dimension
function \(m\mapsto\dim(V\cap W_m)\).  Hence the incidence condition
\[
\dim(V\cap W_{\lambda_i+k-i})\ge i
\]
is equivalent to saying that at least \(i\) of the entries of \(J\) are
\(\ge \lambda_i+k-i\), which is equivalent to
\[
j_{r-i}\ge \lambda_i+k-i.
\]
This is precisely \(J\succeq_k\lambda\).  The disjoint decomposition follows
from uniqueness of exact type.
\end{proof}

\begin{dfn}\label{dfn:Nk-paper}
For \(J=(j_0<\cdots<j_{r-1})\) of length \(r\ge k\), define
\[
N_k(J)=r(r-k)+\sum_{q=0}^{r-1}(j_q-q).
\]
\end{dfn}

\begin{lem}\label{lem:Nk-minimal-paper}
If \(J\succeq_k\lambda\), then
\[
N_k(J)\ge |\lambda|.
\]
The equality holds if and only if
\[
r=k,
\qquad
J=J^{(k)}_\lambda.
\]
\end{lem}

\begin{proof}
Let \(r=|J|\).  From the domination inequalities we obtain
\[
\sum_{i=1}^k\bigl(j_{r-i}-(r-i)\bigr)
\ge
\sum_{i=1}^k\bigl(\lambda_i+k-i-(r-i)\bigr)
=|\lambda|-k(r-k).
\]
For the remaining \(r-k\) terms, the strict increase of \(J\) implies
\[
j_q\ge q
\qquad(q=0,\ldots,r-k-1).
\]
Hence
\[
\sum_{q=0}^{r-1}(j_q-q)
\ge |\lambda|-k(r-k).
\]
Therefore
\[
N_k(J)
=r(r-k)+\sum_{q=0}^{r-1}(j_q-q)
\ge |\lambda|+(r-k)^2
\ge |\lambda|.
\]
The equality forces \(r=k\), and then the domination inequalities must all be
equalities.  Hence
\[
j_{k-i}=\lambda_i+k-i
\qquad(i=1,\ldots,k),
\]
which is exactly \(J=J^{(k)}_\lambda\).  The converse is immediate by direct
calculation.
\end{proof}

\begin{lem}\label{lem:Nk-monotone-paper}
Let \(J=(j_0<\cdots<j_{s-1})\) and \(K=(k_0<\cdots<k_{t-1})\) have lengths
\(s,t\), with \(k\le s\le t\).  Suppose that
\[
k_{t-s+p}\ge j_p
\qquad(p=0,\ldots,s-1).
\]
Then
\[
N_k(K)\ge N_k(J).
\]
If \(K\ne J\), then
\[
N_k(K)>N_k(J).
\]
\end{lem}

\begin{proof}
If \(t=s\), the claim follows immediately from \(k_p\ge j_p\).  If \(t>s\),
then
\[
\sum_{q=0}^{t-1}(k_q-q)
\ge
\sum_{p=0}^{s-1}(j_p-p)-s(t-s),
\]
because the first \(t-s\) entries of \(K\) contribute non-negatively and the
last \(s\) entries dominate the entries of \(J\) after shifting by \(t-s\).
Therefore
\[
\begin{aligned}
N_k(K)-N_k(J)
&=t(t-k)-s(s-k)
  +\sum_{q=0}^{t-1}(k_q-q)-\sum_{p=0}^{s-1}(j_p-p)\\
&\ge t(t-k)-s(s-k)-s(t-s)  \\
&=(t-s)(t-k)>0,
\end{aligned}
\]
since \(t>s\ge k\).
\end{proof}

\begin{prop}\label{prop:local-schur-complement-paper}
Let \(T_0\in\Fred_k\), put \(K=\Ker T_0\), \(Y=\operatorname{Im}T_0\), and set
\(r=\dim K\).  Choose closed subspaces \(X\subset\mathcal H\) and finite-dimensional
\(D\subset\mathcal H\) such that
\[
\mathcal H=K\oplus X,
\qquad
\mathcal H=D\oplus Y.
\]
Then \(\dim D=r-k\), and after shrinking to a neighborhood \(U_0\) of
\(T_0\), every \(T\in U_0\) has a block form
\[
T=
\begin{pmatrix}
a(T)&b(T)\\
c(T)&d(T)
\end{pmatrix}
\colon K\oplus X\longrightarrow D\oplus Y
\]
with \(d(T)\) invertible.  We define \(\tau(T)\) and \(g(T)\) by
\[
\tau(T)=a(T)-b(T)d(T)^{-1}c(T)\in\operatorname{Hom}(K,D),
\]
\[
g(T)=-d(T)^{-1}c(T)\in\operatorname{Hom}(K,X).
\]
Then, for \(\nu\in K\), the vector \(\nu+g(T)\nu\) lies in \(\Ker T\) if and only if
\(\nu\) lies in \(\Ker\tau(T)\). The correspondence \(\nu\mapsto \nu+g(T)\nu\) gives an isomorphism
\[
\Ker\tau(T)\cong\Ker T.
\]
Moreover,
\[
\Coker\tau(T)\cong\Coker T.
\]
\end{prop}

\begin{proof}
Since \(T_0\) is Fredholm, \(Y=\operatorname{Im}T_0\) is closed and
\(\Coker T_0\) is finite-dimensional.  Thus a finite-dimensional complement
\(D\) exists.  The restriction
\[
T_0|_X\colon X\longrightarrow Y
\]
is an isomorphism: it is injective because \(X\cap K=0\), and it is surjective
because \(Y=T_0(\mathcal H)=T_0(X)\).  Hence for \(T\) sufficiently close to
\(T_0\), the block \(d(T)\) is invertible.

For \(u\in K\), \(x\in X\), the equation \(T(u+x)=0\) is
\[
a(T)u+b(T)x=0,
\qquad
c(T)u+d(T)x=0.
\]
The second equation gives \(x=-d(T)^{-1}c(T)u=g(T)u\).  Substitution into the
first equation gives \(\tau(T)u=0\).  This proves the kernel statement.

The cokernel statement follows from the same block calculation.  More explicitly,
precompose \(T\) with the invertible map
\[
K\oplus X\longrightarrow K\oplus X,
\qquad
(u,y)\longmapsto (u,g(T)u+y),
\]
and postcompose with the invertible map
\[
D\oplus Y\longrightarrow D\oplus Y,
\qquad
(\xi,\eta)\longmapsto (\xi-b(T)d(T)^{-1}\eta,\eta).
\]
The resulting operator is
\[
\begin{pmatrix}
\tau(T)&0\\
0&d(T)
\end{pmatrix}.
\]
Since \(d(T)\) is invertible, the cokernel of \(T\) is naturally identified with
\(\Coker\tau(T)\).
\end{proof}

\begin{prop}\label{prop:local-finite-dimensionalization-paper}
Let \(T_0\in\Sigma^{(k),\circ}_{J_0}\), and let \(r=|J_0|=\dim\Ker T_0\).  Then
there exist an open neighborhood \(U\subset\Fred_k\) of \(T_0\), an integer
\(N\), and a smooth map
\[
\Psi_U\colon U\longrightarrow \operatorname{Gr}_r(D\oplus E_N),
\qquad
E_N=\mathcal H/W_N,
\]
where \(\operatorname{Gr}_r(D\oplus E_N)\) denotes the Grassmannian of
\(r\)-dimensional subspaces of \(D\oplus E_N\).  The following hold.

\begin{enumerate}[label=\textup{(\arabic*)}]
\item \(U\cap\Sigma^{(k)}_\lambda=\Psi_U^{-1}(Y^{(r,k)}_\lambda)\),
where \(Y^{(r,k)}_\lambda\) is the Schubert incidence locus in
\(\operatorname{Gr}_r(D\oplus E_N)\) defined by
\[
\dim\bigl(L\cap(0\oplus W_{\lambda_i+k-i}/W_N)\bigr)\ge i
\qquad(i=1,\ldots,k).
\]

\item For every exact type \(J=(j_0<\cdots<j_{s-1})\) appearing in \(U\), one has
\(|J|\le r\) and \(j_{s-1}<N\), and
\[
U\cap\Sigma^{(k),\circ}_{J}=\Psi_U^{-1}(Y^{(r,k),\circ}_{J}),
\]
where \(Y^{(r,k),\circ}_{J}\) is the subset of
\(\operatorname{Gr}_r(D\oplus E_N)\) consisting of all \(L\) such that
\[
\dim\bigl(L\cap(0\oplus E_N)\bigr)=s,
\]
and
\[
\dim\bigl(L\cap(0\oplus W_{j_q}/W_N)\bigr)=s-q,
\qquad
\dim\bigl(L\cap(0\oplus W_{j_q+1}/W_N)\bigr)=s-q-1
\]
for \(q=0,\ldots,s-1\).
\end{enumerate}
\end{prop}

\begin{proof}
We use the notation of \LaterProp{prop:local-schur-complement-paper}.
By \LaterLem{lem:finite-subspace-avoids-WN-paper}, choose \(N\) such that
\[
K\cap W_N=0.
\]
Increasing \(N\) if necessary, assume also that all integers appearing in
\(J_0\) and all \(\lambda_i+k-i\) are smaller than \(N\).  Shrink the
neighborhood so that the map
\[
K\longrightarrow D\oplus E_N,
\qquad
\nu\longmapsto (\tau(T)\nu,\pi_N(\nu+g(T)\nu))
\]
is injective for all \(T\in U\).  This is possible because at \(T=T_0\) the map
is \(\nu\mapsto(0,\pi_N\nu)\), and \(\pi_N|_K\) is injective.

We define
\[
\Psi_U(T)=
\{(\tau(T)\nu,\pi_N(\nu+g(T)\nu))\mid \nu\in K\}
\in \operatorname{Gr}_r(D\oplus E_N).
\]
The smoothness follows from the smooth dependence of \(\tau(T)\) and \(g(T)\).
The injectivity just imposed implies \(\Ker T\cap W_N=0\).  Indeed, if
\(h\in\Ker T\cap W_N\), then by the graph description of the kernel there is
\(\nu\in\Ker\tau(T)\) with \(h=\nu+g(T)\nu\).  Hence
\[
(\tau(T)\nu,\pi_N(\nu+g(T)\nu))=(0,0),
\]
so \(\nu=0\), and therefore \(h=0\).

For every \(j\le N\), with \(W_N/W_N=0\) when \(j=N\), we claim that
\[
\Psi_U(T)\cap(0\oplus W_j/W_N)
=
\{(0,
\pi_N h)\mid h\in\Ker T\cap W_j\}.
\]
The left-hand side consists of elements of the form
\((\tau(T)\nu,\pi_N(\nu+g(T)\nu))\) whose first component is zero and whose
second component lies in \(W_j/W_N\).  The first condition is
\(\nu\in\Ker\tau(T)\), so \(h=\nu+g(T)\nu\in\Ker T\).  Since \(j\le N\),
\(W_N\subset W_j\), and therefore
\[
\pi_N h\in W_j/W_N
\quad\Longleftrightarrow\quad
h\in W_j.
\]
The reverse inclusion is immediate.  Because \(\Ker T\cap W_N=0\), the map
\(h\mapsto\pi_Nh\) is injective on \(\Ker T\cap W_j\).  Hence
\[
\dim(\Ker T\cap W_j)
=
\dim\bigl(\Psi_U(T)\cap(0\oplus W_j/W_N)\bigr).
\]
\end{proof}

\begin{prop}[Local product form of the finite-dimensionalization]\label{prop:local-product-psi-paper}
In the situation of \LaterProp{prop:local-finite-dimensionalization-paper},
after shrinking \(U\), there exist a finite-dimensional vector space \(P\), a
Banach space \(R_0\), an open subset \(\mathcal O\subset P\times R_0\), and a smooth
diffeomorphism
\[
\Theta\colon U\xrightarrow{\cong}\mathcal O
\]
such that \(\Psi_U\) is the composition of \(\Theta\) with the projection
\(P\times R_0\to P\) and a Grassmannian graph chart.  In particular,
\(\Psi_U\) is a submersion onto its finite-dimensional Grassmannian chart.
\end{prop}

\begin{proof}
We use the notation of \LaterProp{prop:local-schur-complement-paper}.  Choose
\(N\) as in \LaterProp{prop:local-finite-dimensionalization-paper}, and
write
\[
E_N=\mathcal H/W_N=K_0\oplus Q,
\qquad K_0=\pi_N(K).
\]
Since \(K\cap W_N=0\), we may choose the complement \(X\) of \(K\) in
\(\mathcal H\) in the form
\[
X=W_N\oplus X_0,
\]
where \(\pi_N|_{X_0}\colon X_0\to Q\) is an isomorphism.  We use this
choice of \(X\) in the Schur-complement coordinates.  A point of the
Grassmannian near the plane \(0\oplus K_0\) is uniquely written as the graph
of a pair
\[
(\alpha,\beta)\in \operatorname{Hom}(K_0,D)\oplus\operatorname{Hom}(K_0,Q).
\]
We denote this finite-dimensional vector space by \(P\).  For \(T\) close to \(T_0\),
the Schur complement data give
\[
\alpha(T)=\tau(T)\circ(\pi_N|_K)^{-1},
\qquad
\beta(T)=\pi_Q\circ\pi_N\circ g(T)\circ(\pi_N|_K)^{-1}.
\]
The remaining variables are the \(W_N\)-component of \(g(T)\), the block
\(b(T)\), and the block \(d(T)\).  We regard these variables as taking values in the Banach space
\[
R_0=
\operatorname{Hom}(K,W_N)\oplus\operatorname{Hom}(X,D)
\oplus\operatorname{Hom}(X,Y).
\]
Let \(V(d_0)\subset \operatorname{Hom}(X,Y)\) be a sufficiently small open
neighborhood of the invertible operator \(d_0=T_0|_X\) consisting of invertible
operators, and set
\[
\mathcal O=P\times\operatorname{Hom}(K,W_N)\times\operatorname{Hom}(X,D)
\times V(d_0)\subset P\times R_0.
\]
After shrinking \(U\), the image lies in \(\mathcal O\).  The map sending \(T\) to
\[
(\alpha(T),\beta(T);\, \pi_{W_N}g(T), b(T), d(T))
\]
is a smooth coordinate map into \(\mathcal O\).  Conversely, given such data,
one reconstructs \(g\), then sets \(c=-dg\) and
\(a=\alpha\circ\pi_N|_K-bg\), obtaining the corresponding block operator.  These
constructions are inverse to each other after shrinking \(U\).  Therefore \(U\) is
smoothly diffeomorphic to \(\mathcal O\), after replacing \(\mathcal O\) by a smaller open subset if necessary.  In these coordinates,
\(\Psi_U\) depends only on the \(P\)-coordinate and is exactly the usual graph chart
map to the Grassmannian.
Thus \(\Psi_U\) is locally a projection followed by a finite-dimensional chart.
\end{proof}

\begin{prop}\label{prop:finite-schubert-strata-paper}
In the finite-dimensional model of \LaterProp{prop:local-finite-dimensionalization-paper},
the stratum \(Y^{(r,k),\circ}_{J}\) is a complex submanifold of
\(\operatorname{Gr}_r(D\oplus E_N)\) of complex codimension \(N_k(J)\).  If
\(L_n\in Y^{(r,k),\circ}_{J}\) and \(L_n\to L\in Y^{(r,k),\circ}_{K}\), then
\(K\) dominates \(J\) in the sense of \LaterLem{lem:Nk-monotone-paper}; in
particular,
\[
N_k(K)\ge N_k(J),
\]
and \(N_k(K)>N_k(J)\) unless \(K=J\).
\end{prop}

\begin{proof}
Let \(s=|J|\).  First impose the condition
\[
\dim(L\cap(0\oplus E_N))=s.
\]
Locally, a subspace \(L\) with this property is described by the data
\[
V=L\cap(0\oplus E_N)\in\operatorname{Gr}_s(E_N),
\]
a subspace \(B\subset D\) of dimension \(r-s\), and a lift of \(B\) to
\((D\oplus E_N)/V\).  Thus this condition has complex codimension
\(s(s-k)\) in \(\operatorname{Gr}_r(D\oplus E_N)\), because \(\dim D=r-k\).
The further condition that \(V\) have exact Schubert type \(J\) inside
\(E_N\) has complex codimension
\[
\sum_{q=0}^{s-1}(j_q-q).
\]
Hence the total complex codimension is
\[
s(s-k)+\sum_{q=0}^{s-1}(j_q-q)=N_k(J).
\]
This also proves the submanifold assertion.

For the closure relation, use upper semicontinuity of intersection dimension:
for every fixed subspace \(F\subset E_N\), the function
\[
L\longmapsto \dim(L\cap(0\oplus F))
\]
is upper semicontinuous on the Grassmannian.  Applying this to
\(F=W_{j_p}/W_N\) gives
\[
\dim\bigl(L\cap(0\oplus W_{j_p}/W_N)\bigr)
\ge s-p.
\]
If \(K=(k_0<\cdots<k_{t-1})\) is the exact type of \(L\), this means
\[
\#\{a\mid k_a\ge j_p\}\ge s-p
\qquad(p=0,\ldots,s-1),
\]
which is equivalent to
\[
k_{t-s+p}\ge j_p.
\]
\LaterLem{lem:Nk-monotone-paper} gives the asserted monotonicity of
\(N_k\).
\end{proof}

\begin{prop}\label{prop:exact-strata-banach-paper}
For every exact type \(J\) of length at least \(k\), the subset
\(\Sigma^{(k),\circ}_{J}\subset\Fred_k\) is a complex Banach submanifold of
complex codimension \(N_k(J)\).
\end{prop}

\begin{proof}
Near any point of \(\Sigma^{(k),\circ}_{J}\), \LaterProp{prop:local-finite-dimensionalization-paper}
identifies the stratum with the inverse image of
\(Y^{(r,k),\circ}_{J}\) under the local map \(\Psi_U\).  By
\LaterProp{prop:local-product-psi-paper}, \(\Psi_U\) is, in local
coordinates, a projection to a finite-dimensional Grassmannian chart followed by
that chart.  Hence it is transverse to every finite-dimensional Schubert stratum
appearing in the chart.  Therefore the inverse image of the finite-dimensional
complex submanifold \(Y^{(r,k),\circ}_{J}\) is a complex Banach submanifold of
the same finite codimension \(N_k(J)\).
\end{proof}

\begin{dfn}\label{dfn:standard-filtration-paper}
The \emph{standard filtration} of \(\Sigma^{(k)}_\lambda\) is
\[
F_0=\Sigma^{(k)}_\lambda,
\]
and, for \(m\ge1\),
\[
F_{2m-1}=F_{2m}
=
\{T\in\Sigma^{(k)}_\lambda\mid N_k(J(T))\ge |\lambda|+m\},
\]
where \(J(T)\) is the exact type of \(T\).
\end{dfn}

\begin{prop}\label{prop:standard-filtration-paper}
The standard filtration consists of closed subsets.  Its strata satisfy
\[
S_{2\ell+1}=\varnothing,
\qquad
S_{2\ell}=\bigsqcup_{\substack{J\succeq_k\lambda\\ N_k(J)=|\lambda|+\ell}}
\Sigma^{(k),\circ}_{J}.
\]
Each \(S_{2\ell}\) is a complex Banach submanifold of complex codimension
\(|\lambda|+\ell\).  This is a filtration by quasi-manifolds.
\end{prop}

\begin{proof}
In a neighborhood supplied by
\LaterProp{prop:local-finite-dimensionalization-paper}, only finitely
many exact types occur, namely those with \(|J|\le r\) and \(j_{|J|-1}<N\).
Thus, in this local model, the subset on which \(N_k(J(T))\ge |\lambda|+m\)
is the product of the inverse image of a finite union of closed subsets of the
form
\[
\{L\in\operatorname{Gr}_r(D\oplus E_N)
\mid
\dim(L\cap(0\oplus W_a/W_N))\ge b\}
\]
with the \(R_0\)-direction in the local splitting.  Hence the union of strata
with \(N_k(J)\ge |\lambda|+m\) is closed.

The formula for the strata in the statement follows from the definition of the
filtration.  If two different exact strata have the same \(N_k\)-value, neither
can occur in the closure of the other, because \LaterProp{prop:finite-schubert-strata-paper}
and \LaterLem{lem:Nk-monotone-paper} would then force strict increase unless
the types were equal.  Therefore, locally, a fixed \(N_k\)-level is a finite
disjoint union of locally closed complex submanifolds of the same codimension,
and hence is itself a complex Banach submanifold.
\end{proof}

\begin{thm}\label{thm:schur-koschorke-quasi-paper}
The subset
\[
\Sigma^{(k)}_\lambda\subset\Fred_k
\]
is a cooriented quasi-manifold of real codimension
\(2|\lambda|\).  Its top stratum is
\[
\Sigma^{(k),\circ}_\lambda
=
\Sigma^{(k),\circ}_{J^{(k)}_\lambda}
=
\left\{
T\in\Fred_k
\ \middle|\
\begin{array}{l}
\dim\Ker T=k,\\
\dim\Coker T=0,\\
\Ker T\text{ has exact type }J^{(k)}_\lambda
\end{array}
\right\}.
\]
\end{thm}

\begin{proof}
By \LaterProp{prop:standard-filtration-paper}, the standard filtration is a filtration by quasi-manifolds.  \LaterLem{lem:Nk-minimal-paper} shows that among all
exact types satisfying the incidence condition, the unique one with minimal
\(N_k\)-value is \(J^{(k)}_\lambda\), and the value is \(|\lambda|\).  Hence the
top stratum is the displayed stratum and has complex codimension
\(|\lambda|\), i.e. real codimension \(2|\lambda|\).

The top stratum is a complex Banach submanifold of the complex Banach manifold
\(\Fred_k\).  Its finite-rank normal bundle is therefore a complex vector bundle,
and hence has a canonical real orientation.  This is the coorientation.
\end{proof}

\begin{dfn}\label{dfn:sigma-k-lambda-class-paper}
The cohomology class associated with \(\Sigma^{(k)}_\lambda\) by
\LaterThm{thm:cibotaru-extension-detailed} is denoted by
\[
[\Sigma^{(k)}_\lambda]
\in H^{2|\lambda|}(\Fred_k).
\]
\end{dfn}

Let
\[
c_i\in H^{2i}(\Fred_k)
\qquad(i\ge1)
\]
be the universal Chern classes normalized by
\[
f^*c_i=c_i(-\operatorname{Ind}(f))
\]
for every finite-dimensional family \(f\) of Fredholm operators.  Here
\(\operatorname{Ind}(f)\in K^0(M)\) denotes the family index of \(f\); when the
kernel and cokernel form vector bundles, it is represented by
\([\Ker f]-[\Coker f]\) \cite{AtiyahHilbert}.  We set
\[
c_0=1,
\qquad
c_i=0\quad(i<0).
\]

\begin{lem}\label{lem:finite-test-detection-paper}
Let \(z\in H^*(\Fred_k)\).  If \(f^*z=0\) for every smooth map
\(f\colon M\to\Fred_k\) from a closed finite-dimensional smooth manifold, then
\(z=0\).
\end{lem}

\begin{proof}
We choose a homotopy equivalence \(h\colon\Fred_k\to BU\) and a homotopy inverse.
Under this equivalence, a nonzero class on \(\Fred_k\) corresponds to a nonzero
polynomial in the universal Chern classes on \(BU\).  By the splitting principle,
such a polynomial is detected after pullback to a finite product of complex
projective spaces \(\prod_i\C P^{N_i}\).  This product is a closed
finite-dimensional smooth manifold.  After composing with a homotopy inverse
\(BU\to\Fred_k\), we obtain a continuous map from \(\prod_i\C P^{N_i}\) to
\(\Fred_k\) detecting the original class.  Since \(\Fred_k\) is an open subset of
the Banach space \(\mathcal B(\mathcal H)\), the standard smooth approximation
theorem for maps into Banach manifolds allows us to replace this map, within its
homotopy class, by a smooth map.  Therefore a class whose pullback vanishes for
all smooth maps from closed finite-dimensional smooth manifolds must be zero.
\end{proof}

\begin{thm}[Giambelli formula for generalized Koschorke classes]\label{thm:schur-koschorke-giambelli-paper}
For every partition \(\lambda=(\lambda_1\ge\cdots\ge\lambda_k\ge0)\),
\[
[\Sigma^{(k)}_\lambda]
=
\det(c_{\lambda_a+b-a})_{1\le a,b\le k}
\in H^{2|\lambda|}(\Fred_k).
\]
\end{thm}

\begin{proof}
By \LaterLem{lem:finite-test-detection-paper}, it is enough to test the
equality after pullback to closed finite-dimensional smooth manifolds.

Let \(f\colon M\to\Fred_k\) be a smooth map from a closed finite-dimensional
smooth manifold.  The exact strata form a countable family of split Banach submanifolds
of finite codimension.  By the Thom--Smale transversality theorem
\cite{AbrahamRobbin,Palais,Smale}, for each exact stratum the set of smooth maps
transverse to that stratum is residual and dense in the strong Whitney
\(C^\infty\)-topology.  The strong Whitney \(C^\infty\)-topology on smooth maps
from the compact manifold \(M\) is a Baire topology; see
\cite{AbrahamRobbin,Palais}.  Hence the intersection of these countably many
residual dense subsets is again residual and dense.  Hence we may choose a smooth map \(f'\colon M\to\Fred_k\),
arbitrarily \(C^0\)-close to \(f\), which is transverse to every exact stratum.
Since \(\Fred_k\) is an open subset of the Banach space
\(\mathcal B(\mathcal H)\), the homotopy from \(f\) to \(f'\) along line segments
stays in \(\Fred_k\) after \(f'\) is chosen sufficiently close.  Thus we may
replace \(f\), within its homotopy class, by a smooth map transverse to every
exact stratum.  After this replacement, the inverse image of any stratum of real
codimension greater than \(\dim M\) is empty.  By the naturality theorem for
classes,
\LaterThm{thm:cibotaru-naturality-detailed}, applied to the class of
the cooriented quasi-manifold \(\Sigma^{(k)}_\lambda\), we obtain
\[
f^*[\Sigma^{(k)}_\lambda]=[f^{-1}(\Sigma^{(k)}_\lambda)]_M.
\]

We work locally on \(M\) using the finite-dimensional model
\(\Psi_U\colon U\to\operatorname{Gr}_r(D\oplus E_N)\).  Over this Grassmannian,
let \(S\) be the tautological bundle and \(Q\) the universal quotient bundle.
For the local family \(f\), define
\[
E_f=\{(x,h)\in M\times\mathcal H\mid f(x)h\in D\}.
\]
The Schur complement construction gives a vector bundle isomorphism
\[
E_f\cong (\Psi_U\circ f)^*S.
\]
The finite-dimensional bundle map
\[
E_f\longrightarrow M\times D,
\qquad
(x,h)\longmapsto (x,f(x)h)
\]
has kernel \(\Ker f(x)\) and cokernel \(\Coker f(x)\), and therefore represents
the family index:
\[
\operatorname{Ind}(f)=[E_f]-[M\times D].
\]
From the exact sequence
\[
0\longrightarrow S\longrightarrow D\oplus E_N\longrightarrow Q\longrightarrow0
\]
we obtain
\[
[(\Psi_U\circ f)^*Q]-[M\times E_N]
=[M\times D]-[E_f]
=-\operatorname{Ind}(f).
\]
Hence
\[
(\Psi_U\circ f)^*c_i(Q)=c_i(-\operatorname{Ind}(f))=f^*c_i.
\]

The Thom class of the top stratum of \(f^{-1}(\Sigma^{(k)}_\lambda)\) is locally
the pullback of the Thom class of the ordinary Schubert variety
\(\Omega_\lambda\) in the finite-dimensional Grassmannian
\cite[Chapter~9]{Fulton}.  The coorientation of
the pulled-back top stratum is the complex orientation induced from this ordinary
Schubert variety.  More precisely, if the local model is
\(\operatorname{Gr}_r(D\oplus E_N)\), then \(\Omega_\lambda\) denotes the ordinary
Schubert variety corresponding to the partition
\[
(\lambda_1,\ldots,\lambda_k,0,\ldots,0)
\]
of length at most \(r\), with respect to the complete flag obtained by extending
the fixed flag in \(E_N\) by the \(D\)-directions.  The Schubert conditions with
indices larger than \(k\) are automatic by dimension counting.  In the usual Giambelli determinant of size \(r\) for Grassmannian Schubert
classes \cite[Chapter~9]{Fulton}, the rows corresponding to the appended zeros
form an upper triangular block with diagonal entries \(c_0(Q)=1\).  Hence the
determinant reduces to the displayed \(k\times k\) determinant.  The finite-dimensional Giambelli formula, in the special Schubert classes
\(c_i(Q)\), gives the following identity \cite[Chapter~9]{Fulton}:
\[
[\Omega_\lambda]
=
\det(c_{\lambda_a+b-a}(Q))_{1\le a,b\le k}.
\]
By pulling back and using the preceding identity for Chern classes, we obtain
\[
f^*[\Sigma^{(k)}_\lambda]
=
\det(f^*c_{\lambda_a+b-a})_{1\le a,b\le k}.
\]
Since such pullbacks detect cohomology, the formula holds on \(\Fred_k\).
\end{proof}

\begin{rem}\label{rem:chern-convention-paper}
The determinant in \LaterThm{thm:schur-koschorke-giambelli-paper} is written
in terms of the classes \(c_i=c_i(-\operatorname{Ind})\).  In the finite-dimensional
Schubert calculation these are the Chern classes of the universal quotient
bundle \(Q\).  Thus the formula uses the standard Grassmannian Schubert
calculus convention.  In terms of ordinary symmetric functions in Chern roots,
this is the dual Jacobi--Trudi convention.
\end{rem}

%%%%%%%%%%%%%%%%%%%%%%%%%%%%%%%%%%%%%%%%%%%%%%%%
\subsection{Generalized Koschorke classes}
%%%%%%%%%%%%%%%%%%%%%%%%%%%%%%%%%%%%%%%%%%%%%%%%

We now transport the classes from \(\Fred_k\) to \(\Fredzero\).  This produces the
partition-indexed generalized Koschorke classes used in Section~5.

We choose the standard unilateral shift operators on a fixed model
\(\mathcal H=\ell^2(\mathbb N_0)\):
\[
L_ke_j=
\begin{cases}
0,&0\le j<k,\\
e_{j-k},&j\ge k,
\end{cases}
\qquad
R_ke_j=e_{j+k}.
\]
Then \(L_k\) is surjective Fredholm of index \(k\), \(R_k\) is injective
Fredholm of index \(-k\), and
\[
L_kR_k=I.
\]

\begin{prop}\label{prop:shift-equivalence-paper}
The map
\[
\sigma_k\colon\Fredzero\longrightarrow\Fred_k,
\qquad
\sigma_k(T)=L_kT,
\]
is a homotopy equivalence.  A homotopy inverse is
\[
\rho_k\colon\Fred_k\longrightarrow\Fredzero,
\qquad
\rho_k(A)=R_kA.
\]
Moreover,
\[
\sigma_k^*c_i=c_i
\]
for the universal classes normalized by \(c_i=c_i(-\operatorname{Ind})\).
\end{prop}

\begin{proof}
Since \(L_kR_k=I\), one has
\[
\sigma_k\rho_k(A)=A
\qquad(A\in\Fred_k).
\]
On \(\Fredzero\),
\[
\rho_k\sigma_k(T)=R_kL_kT=P_kT,
\]
where \(P_k=R_kL_k\) is the projection onto
\(\overline{\operatorname{span}}\{e_k,e_{k+1},\ldots\}\).  If \(Q_k\) denotes the
projection onto \(\operatorname{span}\{e_0,\ldots,e_{k-1}\}\), then
\[
P_k=I-Q_k.
\]
The path
\[
P_{k,t}=I-tQ_k
\qquad(0\le t\le1)
\]
consists of Fredholm operators of index zero and connects \(I\) to \(P_k\).
Thus
\[
H(t,T)=P_{k,t}T
\]
is a homotopy from \(\operatorname{id}_{\Fredzero}\) to \(\rho_k\sigma_k\).
Hence \(\sigma_k\) is a homotopy equivalence.

If \(f\colon M\to\Fredzero\) is a continuous map, then
\[
\operatorname{Ind}(\sigma_k\circ f)=\operatorname{Ind}(f)+[\underline{\C}^k].
\]
Adding a trivial bundle does not change the Chern classes of
\(-\operatorname{Ind}(f)\).  Hence \(\sigma_k^*c_i=c_i\).
\end{proof}

\begin{dfn}\label{dfn:generalized-koschorke-lambda-paper}
Let \(\lambda\) be a partition.  Choose \(k\ge \ell(\lambda)\), append trailing
zeros to view \(\lambda\) as a partition of length \(k\), and define
\[
[\Sigma_\lambda]
:=
\sigma_k^*[\Sigma^{(k)}_\lambda]
\in H^{2|\lambda|}(\Fredzero).
\]
We call \([\Sigma_\lambda]\) the \emph{generalized Koschorke class}, associated with \(\lambda\).
\end{dfn}

\begin{cor}[Giambelli formula for generalized Koschorke classes]\label{cor:generalized-koschorke-giambelli-paper}
The class \([\Sigma_\lambda]\) is independent of the choice of
\(k\ge\ell(\lambda)\), and
\[
[\Sigma_\lambda]
=
\det(c_{\lambda_a+b-a})_{1\le a,b\le k}
\in H^{2|\lambda|}(\Fredzero).
\]
As \(\lambda\) ranges over all partitions, the classes \([\Sigma_\lambda]\)
form an additive basis of \(H^*(\Fredzero)\).
\end{cor}

\begin{proof}
The determinant formula follows from \LaterThm{thm:schur-koschorke-giambelli-paper}
and \LaterProp{prop:shift-equivalence-paper}.  If \(k\) is increased by
one and a trailing zero is appended to \(\lambda\), the determinant is enlarged
by adding a final row whose only nonzero entry is the last entry \(c_0=1\).
Therefore the determinant is unchanged.  This proves the independence of \(k\).
The determinant classes are precisely the Schubert basis expressed in terms of the special
Schubert generators of \(H^*(BU)\cong H^*(\Fredzero)\), and hence form an
additive basis.
\end{proof}

\begin{cor}\label{cor:ordinary-koschorke-pq-paper}
Let \(p,q\ge1\), and let \(\lambda=(p^q)\) be the rectangular partition of
length \(q\).  Then
\[
[\Sigma_{(p^q)}]
=
\det(c_{p-i+j})_{1\le i,j\le q}
\in H^{2pq}(\Fredzero).
\]
Under the standard identification of the cohomology rings of the connected
components of the space of Fredholm operators, this recovers Koschorke's universal class
\[
k_{p,q}\in H^{2pq}(\Fred_{p-q}).
\]
Consequently, for a family \(A\colon X\to\Fred_{p-q}\),
\[
\chi_{p,q}(A)
=
\det\bigl(c_{p-i+j}(-\operatorname{Ind}(A))\bigr)_{1\le i,j\le q}.
\]
\end{cor}

\begin{proof}
Apply \LaterCor{cor:generalized-koschorke-giambelli-paper} to the
partition \(\lambda=(p,\ldots,p)\) of length \(q\).  The determinant becomes
\[
\det(c_{p-i+j})_{1\le i,j\le q}.
\]
This is exactly Koschorke's determinantal formula recalled in
\LaterThm{thm:even-koschorke-determinant}.  Pulling back along a family \(A\) gives the final expression.
\end{proof}

\begin{rem}\label{rem:actual-fredzero-locus-paper}
If one wants an actual subset of \(\Fredzero\) representing \([\Sigma_\lambda]\),
one can use
\[
\widetilde\Sigma^{(0,k)}_\lambda
=
\sigma_k^{-1}(\Sigma^{(k)}_\lambda)
=
\left\{
T\in\Fredzero
\ \middle|\
\dim_\C(\Ker(L_kT)\cap W_{\lambda_i+k-i})\ge i
\quad(i=1,\ldots,k)
\right\}.
\]
\end{rem}

%%%%%%%%%%%%%%%%%%%%%%%%%%%%%%%%%%%%%%%%%%%%%%%%
\subsection{Generalized odd Koschorke classes}
%%%%%%%%%%%%%%%%%%%%%%%%%%%%%%%%%%%%%%%%%%%%%%%%

We now define the odd analogue in the space of self-adjoint Fredholm operators.  This part
uses Cibotaru's Lagrangian Grassmannian model for odd \(K\)-theory and his
cooriented Schubert quasi-manifolds \cite{Cibotaru}.

Let
\[
\widehat{\mathcal H}=\mathcal H\oplus\mathcal H,
\qquad
\mathcal H_+=\mathcal H\oplus0,
\qquad
\mathcal H_- =0\oplus\mathcal H.
\]
Equip \(\widehat{\mathcal H}\) with the unitary operator
\[
J(u,v)=(v,-u).
\]
A closed complex subspace \(L\subset\widehat{\mathcal H}\) is called
\emph{Lagrangian} if
\[
JL=L^\perp.
\]
The set of all Lagrangian subspaces is denoted by
\(\operatorname{Lag}(\widehat{\mathcal H})\).  Following Cibotaru, we set
\[
\operatorname{Lag}^{-}(\widehat{\mathcal H})
=
\{L\in\operatorname{Lag}(\widehat{\mathcal H})
\mid (L,\mathcal H_-)\text{ is a Fredholm pair}\},
\]
where this means that
\[
\dim(L\cap\mathcal H_-)<\infty,
\qquad
L+\mathcal H_-\text{ is closed},
\qquad
\operatorname{codim}(L+\mathcal H_-)<\infty.
\]

For the fixed flag \(W_j\subset\mathcal H\), we set
\[
\widehat W_j=0\oplus W_j\subset\mathcal H_-.
\]
For \(T\in\Fredsaone\), define its switched graph by
\[
\widetilde\Gamma_T=\{(Tv,v)\mid v\in\mathcal H\}\subset\widehat{\mathcal H}.
\]
Because \(T\) is self-adjoint, \(\widetilde\Gamma_T\) is Lagrangian; because
\(T\) is Fredholm, the pair \((\widetilde\Gamma_T,\mathcal H_-)\) is Fredholm.
Thus one obtains the switched-graph map
\[
\Gamma^{\mathrm{sw}}\colon\Fredsaone\longrightarrow
\operatorname{Lag}^{-}(\widehat{\mathcal H}),
\qquad
T\longmapsto\widetilde\Gamma_T.
\]

Let
\[
J=(j_0<j_1<\cdots<j_{k-1})
\]
be a strictly increasing sequence of non-negative integers, and put
\[
N_J^{\mathrm{odd}}=
\sum_{a=0}^{k-1}(2j_a+1).
\]

\begin{dfn}\label{dfn:odd-generalized-sigma-paper}
The generalized odd Koschorke degeneracy set associated with \(J\) is
\[
\Sigma_J^{\mathrm{odd}}
=
\left\{
T\in\Fredsaone
\ \middle|\
\dim_\C(\Ker T\cap W_{j_a})\ge k-a
\quad(a=0,\ldots,k-1)
\right\}.
\]
\end{dfn}

Cibotaru's Schubert locus \(\Omega_J^{\mathrm{sa}}\) in the Lagrangian
Grassmannian is defined by
\[
\dim_\C(L\cap \widehat W_{j_a})\ge k-a
\qquad(a=0,\ldots,k-1).
\]
Since
\[
\widetilde\Gamma_T\cap \widehat W_j
\cong
\Ker T\cap W_j,
\]
one has
\[
\Sigma_J^{\mathrm{odd}}=(\Gamma^{\mathrm{sw}})^{-1}(\Omega_J^{\mathrm{sa}}).
\]

\begin{thm}[Generalized odd Koschorke classes]\label{thm:odd-generalized-classes-paper}
The subset \(\Sigma_J^{\mathrm{odd}}\subset\Fredsaone\) is a cooriented
quasi-manifold of real codimension \(N_J^{\mathrm{odd}}\), and its class is
\[
[\Sigma_J^{\mathrm{odd}}]
=
 c_{j_0+1/2}c_{j_1+1/2}\cdots c_{j_{k-1}+1/2}
\in H^{N_J^{\mathrm{odd}}}(\Fredsaone),
\]
with the factors taken in the displayed order.
\end{thm}

\begin{proof}
Cibotaru proves, in the Lagrangian Grassmannian \(\operatorname{Lag}^{-}\),
that \(\Omega_J^{\mathrm{sa}}\) is a cooriented quasi-manifold and that its class
is the ordered product of the corresponding primitive odd generators
\cite{Cibotaru}.  Let
\(\mathcal U\subset\operatorname{Lag}^{-}(\widehat{\mathcal H})\) be the open
subset consisting of Lagrangians for which the projection to the second factor
\(\widehat{\mathcal H}=\mathcal H\oplus\mathcal H\to\mathcal H\) is an
isomorphism.  The switched graph construction gives a smooth diffeomorphism
from the whole space of bounded self-adjoint Fredholm operators onto
\(\mathcal U\): its inverse sends a Lagrangian \(L\in\mathcal U\) to the
operator \(T\) defined by \((Ty,y)\in L\).  The Lagrangian condition gives
self-adjointness of \(T\), and the Fredholm-pair condition gives the Fredholm
property.

We now restrict this diffeomorphism to the component used in odd \(K\)-theory.
Set
\[
\mathcal U^{1}=\Gamma^{\mathrm{sw}}(\Fredsaone)\subset\mathcal U.
\]
Then
\[
\Gamma^{\mathrm{sw}}\colon\Fredsaone\xrightarrow{\cong}\mathcal U^{1}
\]
is a smooth diffeomorphism.  Hence no transversality assertion for
\(\Gamma^{\mathrm{sw}}\) is needed; we are only restricting Cibotaru's
quasi-manifold to the open component \(\mathcal U^{1}\) and transporting it by
a diffeomorphism, exactly as in
\LaterThm{thm:cibotaru-open-diffeomorphism-detailed}.

Under this diffeomorphism,
\[
\widetilde\Gamma_T\cap\widehat W_j\cong\Ker T\cap W_j,
\]
so the restricted locus is precisely \(\Sigma_J^{\mathrm{odd}}\).  Moreover the
primitive generator indexed by \(i\) pulls back to the universal odd Chern class
\(c_{i+1/2}\).  Hence
\[
[\Sigma_J^{\mathrm{odd}}]
=(\Gamma^{\mathrm{sw}})^*[\Omega_J^{\mathrm{sa}}]
=c_{j_0+1/2}\cdots c_{j_{k-1}+1/2}.
\]
\end{proof}

\begin{cor}\label{cor:odd-koschorke-product-paper}
For every \(p\ge1\),
\[
k_p^{\mathrm{odd}}
=
 c_{1/2}c_{3/2}\cdots c_{(2p-1)/2}.
\]
Thus \LaterThm{thm:odd-koschorke-chern} holds.
\end{cor}

\begin{proof}
We take
\[
J=(0,1,\ldots,p-1).
\]
Then \(\Sigma_J^{\mathrm{odd}}\) is the closed set
\[
Z_p=\{T\in\Fredsaone\mid \dim\Ker T\ge p\}.
\]
Indeed, if \(K\subset\mathcal H\) is finite-dimensional and \(\dim K\ge p\),
then for \(a=0,\ldots,p-1\) one has
\[
\dim(K\cap W_a)\ge \dim K-a\ge p-a,
\]
because \(W_a\) has complex codimension \(a\).  Conversely, the condition for
\(a=0\) gives \(\dim K\ge p\).

We compare the two classes after restriction to
\[
V_p^{\mathrm{sa}}=\{T\in\Fredsaone\mid \dim\Ker T\le p\}.
\]
On this open subset,
\[
Z_p\cap V_p^{\mathrm{sa}}=A_p^{\mathrm{sa}}.
\]
The Cibotaru Schubert filtration may still refine \(A_p^{\mathrm{sa}}\) by the
position of the kernel relative to the fixed flag.  Its top piece inside
\(V_p^{\mathrm{sa}}\) is the flag-generic open part
\[
A_{p,\mathrm{gen}}^{\mathrm{sa}}
=
\{T\in A_p^{\mathrm{sa}}\mid \Ker T\text{ has exact type }(0,1,\ldots,p-1)\}.
\]
Equivalently, if \(K=\Ker T\), then
\[ 
\dim(K\cap W_a)=p-a
\qquad(a=0,\ldots,p).
\]
The complement
\[
B_p=A_p^{\mathrm{sa}}\setminus A_{p,\mathrm{gen}}^{\mathrm{sa}}
\]
is a union of positive Schubert-codimension pieces in the kernel Grassmannian.
Consequently every stratum of \(B_p\), viewed as a subset of
\(V_p^{\mathrm{sa}}\), has real codimension at least \(p^2+2\).  Hence, by the same local-duality argument as in \LaterLem{lem:relative-vanishing-vp},
\[
H^{p^2}(V_p^{\mathrm{sa}},V_p^{\mathrm{sa}}\setminus B_p)=0.
\]
Thus two degree-\(p^2\) classes on \(V_p^{\mathrm{sa}}\) are equal if their
restrictions to \(V_p^{\mathrm{sa}}\setminus B_p\) are equal.

On \(V_p^{\mathrm{sa}}\setminus B_p\), the restricted class is the class
associated with the cooriented submanifold
\[
A_{p,\mathrm{gen}}^{\mathrm{sa}}
\subset
V_p^{\mathrm{sa}}\setminus B_p.
\]
The relative class defining \(e k_p^{\mathrm{odd}}\), after restriction to the
same open subset, is also the class associated with this same submanifold.  Therefore
\([\Sigma_J^{\mathrm{odd}}]|_{V_p^{\mathrm{sa}}}\) and
\(k_p^{\mathrm{odd}}|_{V_p^{\mathrm{sa}}}\) agree once their coorientations are
identified.

We spell out this coorientation comparison.  Let \(T_0\in A_p^{\mathrm{sa}}\) and
set \(K=\Ker T_0\).  In a small self-adjoint Fredholm chart determined by the
orthogonal decomposition \(\mathcal H=K\oplus K^\perp\), a nearby operator has
block form
\[
T=\begin{pmatrix}a&b^*\\ b&d\end{pmatrix},
\]
with \(d\) invertible, and its kernel is governed by the Schur complement
\[
s(T)=a-b^*d^{-1}b\in\operatorname{Herm}(K).
\]
Thus the normal coordinate to the kernel-dimension stratum is the Hermitian form
\(s(T)\).  Under the switched graph map to Cibotaru's Lagrangian Grassmannian
model, the same first-order coordinate is obtained by projecting the graph to
\(K\oplus K^*\); the resulting quadratic form on \(K\) is again \(s(T)\).  Hence
the comparison of normal spaces is the identity on
\(\operatorname{Herm}(K)\).  The standard orientation of
\(\operatorname{Herm}(K)\) is also invariant under a change of unitary frame,
because \(U(K)\) acts by \(H\mapsto uHu^*\) and \(U(K)\) is connected.
Therefore Cibotaru's coorientation and the coorientation used in the
definition of \(k_p^{\mathrm{odd}}\) agree.  Consequently, the two
classes restrict to the same class in
\[
H^{p^2}(V_p^{\mathrm{sa}}).
\]

It remains to see that equality after this restriction implies equality on
\(\Fredsaone\).  By \LaterLem{lem:relative-cohomology-fred-vp}, applied with
\(p+1\) in place of \(p\), and by \LaterLem{lem:relative-vanishing-vp}, we have
\[
H^{p^2}(\Fredsaone,V_p^{\mathrm{sa}})=0,
\]
since \(p^2<(p+1)^2\).  Therefore the restriction map
\[
H^{p^2}(\Fredsaone)\longrightarrow H^{p^2}(V_p^{\mathrm{sa}})
\]
is injective.  Thus the class of \(\Sigma_J^{\mathrm{odd}}\) is equal to
the odd Koschorke class \(k_p^{\mathrm{odd}}\) on \(\Fredsaone\).
The product formula now follows from \LaterThm{thm:odd-generalized-classes-paper}.
\end{proof}

%%%%%%%%%%%%%%%%%%%%%%%%%%%%%%%%%%%%%%%%%%%%%%%%
\section{Virasoro and Neveu--Schwarz interpretations}
%%%%%%%%%%%%%%%%%%%%%%%%%%%%%%%%%%%%%%%%%%%%%%%%

This section recalls the concepts from representation theory used below.  We first
recall the Virasoro algebra and the bosonic Fock representation needed for the
even Koschorke classes.  Atiyah--Segal suggested a relation with Koschorke
classes in their work on twisted \(K\)-theory and cohomology, and we recall
Gomi's explicit description of this relation \cite{AtiyahSegal,GomiSugaku}.  We then recall the
Neveu--Schwarz algebra, viewed as the \(N=1\) super-Virasoro algebra in the
Neveu--Schwarz sector, and formulate the corresponding statement for the
cohomology of \(\Fredzero\times\Fredsaone\).

%%%%%%%%%%%%%%%%%%%%%%%%%%%%%%%%%%%%%%%%%%%%%%%%
\subsection{Virasoro algebra}
%%%%%%%%%%%%%%%%%%%%%%%%%%%%%%%%%%%%%%%%%%%%%%%%

\begin{dfn}[Virasoro algebra]\label{dfn:virasoro-algebra-paper}
The Virasoro algebra is the complex Lie algebra generated by elements
\(L_n\) \((n\in\Z)\) and a central element \(C\), subject to
\[
[L_m,L_n]=(m-n)L_{m+n}+\frac{m^3-m}{12}\delta_{m+n,0}C,
\qquad [C,L_n]=0.
\]
It is an infinite-dimensional Lie algebra and is the universal central extension
of the Witt algebra.  If \(C\) acts on a module by the scalar \(c\), then \(c\) is called the central charge.  We use the
standard conventions for Virasoro modules and Fock modules; see
\cite{DiFrancescoMathieuSenechal,TsuchiyaKanie,WakimotoYamadaI,WakimotoYamadaII}.
\end{dfn}

\begin{dfn}[Singular vector]\label{dfn:virasoro-singular-paper}
Let \(M\) be a Virasoro module.  A vector \(v\in M\) is called singular if
\[
L_n v=0\qquad(n>0).
\]
If the module is graded, the degree of a homogeneous singular vector is its
graded degree.
\end{dfn}

Let
\[
A=\C[x_1,x_2,x_3,\ldots],
\qquad \operatorname{wt}(x_n)=n.
\]
Let \(\beta=\sqrt2\).  The Heisenberg algebra is generated by
\(a_n\) \((n\in\Z)\) with
\[
[a_m,a_n]=m\delta_{m+n,0}.
\]
It acts on \(A\) by
\[
a_n=(-1)^{n-1}\frac{\beta}{(n-1)!}\frac{\partial}{\partial x_n},
\qquad
 a_{-n}=(-1)^{n-1}\frac{n!}{\beta}x_n
\quad(n>0),
\]
and \(a_0=0\).  The vacuum vector is \(1\in A\), and it is annihilated by all
\(a_n\) with \(n\ge0\).

\begin{dfn}[Normal ordering]\label{dfn:normal-ordering-paper}
For generators of the Heisenberg algebra, the normal ordered product \(:a_m a_n:\) is obtained
by moving the creation operators \(a_{-r}\) \((r>0)\) to the left of the
annihilation operators \(a_r\) \((r>0)\).  Thus
\[
:a_m a_n:
=
\begin{cases}
a_m a_n,& m\le 0,\\
a_n a_m,& m>0.
\end{cases}
\]
For products containing fermionic generators below, the same convention is used,
with the Koszul sign inserted when two odd generators are interchanged.
\end{dfn}

\begin{prop}[\cite{TsuchiyaKanie,WakimotoYamadaI,WakimotoYamadaII}]\label{prop:bosonic-fock-virasoro-paper}
The operators
\[
L_n=\frac12\sum_{m\in\Z}:a_m a_{n-m}:
\]
define on \(A\) a Virasoro representation of central charge \(1\).
\end{prop}

\begin{proof}
This is the standard Sugawara construction for the Heisenberg algebra.  With the
normalization above, the Heisenberg relations imply
\[
[L_m,L_n]=(m-n)L_{m+n}+\frac{m^3-m}{12}\delta_{m+n,0},
\]
and hence the central charge is \(1\); see
\cite{TsuchiyaKanie,WakimotoYamadaI,WakimotoYamadaII}.
\end{proof}

The singular vectors relevant here are
those indexed by rectangular partitions.  We use the following form of the
Mimachi--Yamada formula.

\begin{thm}[Mimachi--Yamada~\cite{MimachiYamada}]\label{thm:mimachi-yamada-gomi-paper}
For each \(r\ge0\), there is, up to a scalar, a unique singular vector
\(\chi_{r,r}\in A\) of degree \(r^2\).  Under the power-sum notation
\[
p_n=n!x_n,
\]
this singular vector is represented by the Schur function of rectangular shape
\((r^r)\):
\[
\chi_{r,r}=s_{(r^r)}(p_1,p_2,p_3,\ldots).
\]
\end{thm}

\begin{proof}
The existence and uniqueness of the rectangular singular vector in this Fock
module, together with its expression by the Schur function of shape \((r^r)\),
is the rectangular case of the formula of Mimachi--Yamada
\cite{MimachiYamada}.
\end{proof}

%%%%%%%%%%%%%%%%%%%%%%%%%%%%%%%%%%%%%%%%%%%%%%%%
\subsection{Koschorke classes and singular vectors in a Virasoro representation}
%%%%%%%%%%%%%%%%%%%%%%%%%%%%%%%%%%%%%%%%%%%%%%%%

Over \(\C\), we identify
\[
H^*(\Fredzero;\C)\cong \C[x_1,x_2,x_3,\ldots],
\]
where \(x_n\) is the degree \(2n\) component of the universal Chern character.
Equivalently, if \(p_n\) denotes the \(n\)-th power sum in the Chern roots, then
\[
p_n=n!x_n.
\]

Atiyah--Segal suggested a connection between Koschorke classes and the
Virasoro algebra in their study of twisted \(K\)-theory and cohomology
\cite{AtiyahSegal}.  Gomi made this connection explicit by combining the
Mimachi--Yamada formula with Koschorke's determinantal expression for the
classes \(k_{r,r}\) \cite{GomiSugaku,Koschorke,MimachiYamada}.

\begin{thm}[Gomi~\cite{GomiSugaku}, in the present normalization]\label{thm:gomi-virasoro-koschorke-paper}
Under the identification
\[
H^*(\Fredzero;\C)\cong \C[x_1,x_2,x_3,\ldots],
\]
the singular vector \(\chi_{r,r}\) of degree \(r^2\) corresponds to the
Koschorke class
\[
k_{r,r}=\det(c_{r-i+j})_{1\le i,j\le r}\in H^{2r^2}(\Fredzero;\C).
\]
Moreover, the generator \(L_1\) of the Virasoro algebra satisfies
\[
L_1=-D_{\mathrm{ev}},
\]
where \(D_{\mathrm{ev}}\) is the derivation defined in Section~3.7.
\end{thm}

\begin{proof}
By \LaterThm{thm:mimachi-yamada-gomi-paper}, the singular vector
\(\chi_{r,r}\) is the Schur function \(s_{(r^r)}\) in the Chern roots.  Expressing
this Schur function in terms of the elementary symmetric functions gives
\[
s_{(r^r)}=\det(c_{r-i+j})_{1\le i,j\le r}
\]
by the Jacobi--Trudi identity in the dual convention used for Chern classes
\cite{Macdonald}.  Koschorke's formula identifies this determinant with the
Koschorke class \(k_{r,r}\) \cite{Koschorke}.  The equality \(L_1=-D_{\mathrm{ev}}\)
follows by applying the definition of the Heisenberg operators to the formula
\(L_1=\frac12\sum_{m\in\Z}:a_m a_{1-m}:\).  The only nonzero contributions are
\(m\le0\) and \(1-m>0\), and the normalization of the operators \(a_n\) gives
\[
L_1=-\sum_{n\ge1}x_n\frac{\partial}{\partial x_{n+1}}
=-D_{\mathrm{ev}}.
\]
\end{proof}

\begin{rem}\label{rem:L1-geometry-paper}
The derivation \(D_{\mathrm{ev}}\) has a geometric meaning: it is the
infinitesimal operation coming from tensoring by a line bundle, as explained by
Atiyah--Segal \cite{AtiyahSegal}.  In view of Nakajima's geometric
constructions of Heisenberg and Virasoro actions \cite{Nakajima}, it is natural
to ask whether all generators \(L_n\) of the Virasoro algebra admit a geometric
realization in topological \(K\)-theory.  This question is also implicit in Gomi's discussion \cite{GomiSugaku}. It remains
open.
\end{rem}

%%%%%%%%%%%%%%%%%%%%%%%%%%%%%%%%%%%%%%%%%%%%%%%%
\subsection{Neveu--Schwarz algebra}
%%%%%%%%%%%%%%%%%%%%%%%%%%%%%%%%%%%%%%%%%%%%%%%%

\begin{dfn}[Neveu--Schwarz algebra]\label{dfn:NS-algebra-paper}
The Neveu--Schwarz algebra is the complex Lie superalgebra with even generators
\(L_n\) \((n\in\Z)\), odd generators \(G_r\) \((r\in\Z+1/2)\), and a central element
\(C\), subject to
\[
[L_m,L_n]=(m-n)L_{m+n}+\frac{C}{12}(m^3-m)\delta_{m+n,0},
\]
\[
[L_m,G_r]=\left(\frac m2-r\right)G_{m+r},
\]
\[
\{G_r,G_s\}=2L_{r+s}+\frac{C}{3}\left(r^2-\frac14\right)\delta_{r+s,0}.
\]
In this paper we use representations in which \(C\) acts by \(3/2\).  For the
standard representation theory of the Neveu--Schwarz algebra, see
\cite{DiFrancescoMathieuSenechal,KacWakimoto,IoharaKogaI,IoharaKogaII}.
\end{dfn}

A vector in a Neveu--Schwarz module is called singular if it is annihilated by
all positive modes \(L_n\) \((n>0)\) and \(G_r\) \((r>0)\).  We use only the
polynomial representation described below, so every vector under consideration
is a finite polynomial.

The result for the Neveu--Schwarz algebra in this paper is more modest than
the even statement for the Virasoro algebra above.  We construct a fixed representation on cohomology and use the bases from
Section~4, formed by generalized Koschorke classes, to expand singular vectors
in that representation.

%%%%%%%%%%%%%%%%%%%%%%%%%%%%%%%%%%%%%%%%%%%%%%%%
\subsection{The cohomological Neveu--Schwarz representation}
%%%%%%%%%%%%%%%%%%%%%%%%%%%%%%%%%%%%%%%%%%%%%%%%

Set
\[
F_{\mathrm{ev}}=\Fredzero,
\qquad
F_{\mathrm{odd}}=\Fredsaone.
\]
Then
\[
\mathcal F
=
H^*(F_{\mathrm{ev}}\times F_{\mathrm{odd}};\C)
\cong
\C[x_1,x_2,\ldots]
\otimes
\Lambda_{\C}(y_{1/2},y_{3/2},y_{5/2},\ldots).
\]
Here \(x_n\) is the degree \(2n\) component of the universal Chern character on
\(\Fredzero\).  The odd variable \(y_{m+1/2}\) is the degree \(2m+1\) component of
the universal odd Chern character on \(\Fredsaone\).  With the orientation
convention fixed in Sections~3 and~4, the complex image of the odd Chern class
is
\[
c_{m+1/2}=m!\,y_{m+1/2}
\qquad(m\ge0).
\]
We use the weights
\[
\operatorname{wt}(x_n)=n,
\qquad
\operatorname{wt}(y_{m+1/2})=m+\frac12.
\]

Define Heisenberg operators \(a_n\) and Clifford operators \(b_r\) on
\(\mathcal F\) by
\[
a_{-n}=(-1)^{n-1}\frac{n!}{\beta}x_n,
\qquad
 a_n=(-1)^{n-1}\frac{\beta}{(n-1)!}\frac{\partial}{\partial x_n}
\qquad(n>0),
\]
\[
b_{-(m+1/2)}=(-1)^m m!y_{m+1/2},
\qquad
 b_{m+1/2}=\frac{(-1)^m}{m!}\frac{\partial}{\partial y_{m+1/2}}
\qquad(m\ge0),
\]
and put \(a_0=0\).  The derivative with respect to an odd variable is the left
derivative on the exterior algebra.  These operators satisfy
\[
[a_m,a_n]=m\delta_{m+n,0},
\qquad
\{b_r,b_s\}=\delta_{r+s,0}.
\]

\begin{dfn}\label{dfn:super-operators-paper}
On \(\mathcal F\), define
\[
L_n=\frac12\sum_{m\in\Z}:a_m a_{n-m}:
+\frac14\sum_{r\in\Z+1/2}(n-2r):b_r b_{n-r}:,
\]
and
\[
G_r=\sum_{m\in\Z}a_m b_{r-m}.
\]
The normal ordering is the one fixed in \LaterDfn{dfn:normal-ordering-paper}.
On each polynomial vector, only finitely many summands act nontrivially.
\end{dfn}

\begin{prop}[\cite{DiFrancescoMathieuSenechal,IoharaKogaII,DLM}]\label{prop:super-representation-paper}
The operators in \LaterDfn{dfn:super-operators-paper} define on \(\mathcal F\) a
representation of the Neveu--Schwarz algebra with central charge \(3/2\).
\end{prop}

\begin{proof}
The first summand in \(L_n\) is the bosonic Sugawara operator and gives central
charge \(1\).  The second summand is the standard fermionic contribution and
gives central charge \(1/2\).  The two parts commute.  The operators
\[
G_r=\sum_{m\in\Z}a_m b_{r-m}
\]
are the standard supercurrents for these oscillators.  The relations in
\LaterDfn{dfn:NS-algebra-paper} follow from the Heisenberg and Clifford
relations; see \cite{DiFrancescoMathieuSenechal,IoharaKogaII,DLM}.
\end{proof}

Let \(D_{\mathrm{ev}}\) and \(D_{\mathrm{odd}}\) be the derivations introduced in
Section~3.7, extended to the polynomial-exterior algebra \(\mathcal F\).  Define
\[
D_{\mathcal F}=D_{\mathrm{ev}}+D_{\mathrm{odd}},
\]
or explicitly,
\[
D_{\mathcal F}
=
\sum_{n\ge1}x_n\frac{\partial}{\partial x_{n+1}}
+
\sum_{m\ge0}y_{m+1/2}\frac{\partial}{\partial y_{m+3/2}}.
\]

\begin{prop}\label{prop:L1-equals-minus-D-paper}
On \(\mathcal F\), one has
\[
L_1=-D_{\mathcal F}.
\]
\end{prop}

\begin{proof}
We compute directly from the definitions of the modes.  In the even part, the
same computation as in \LaterThm{thm:gomi-virasoro-koschorke-paper} gives
\[
L_1^{\mathrm{ev}}
=-\sum_{n\ge1}x_n\frac{\partial}{\partial x_{n+1}}.
\]
For the fermionic part, the normal ordered expression gives
\[
L_1^{\mathrm{odd}}
=-\sum_{m\ge1}y_{m-1/2}\frac{\partial}{\partial y_{m+1/2}}
=-\sum_{m\ge0}y_{m+1/2}\frac{\partial}{\partial y_{m+3/2}}.
\]
Adding the two contributions yields \(L_1=-D_{\mathcal F}\).
\end{proof}

%%%%%%%%%%%%%%%%%%%%%%%%%%%%%%%%%%%%%%%%%%%%%%%%
\subsection{Power-sum coordinates and a basis formed by tensor products}
%%%%%%%%%%%%%%%%%%%%%%%%%%%%%%%%%%%%%%%%%%%%%%%%

We use the following power-sum coordinates for the polynomial representation.
Let
\[
R=\C[p_1,p_2,\ldots]\otimes\Lambda_\C(\widetilde p_0,\widetilde p_1,\ldots),
\]
with
\[
\operatorname{wt}(p_n)=n,
\qquad
\operatorname{wt}(\widetilde p_m)=m+\frac12.
\]
Define a graded-commutative algebra isomorphism
\[
\Phi\colon R\longrightarrow\mathcal F
\]
by
\[
\Phi(p_n)=n!x_n,
\qquad
\Phi(\widetilde p_m)=m!y_{m+1/2}.
\]

\begin{prop}\label{prop:Phi-intertwines-paper}
The map \(\Phi\) intertwines the Neveu--Schwarz actions on \(R\) and on
\(\mathcal F\), where the action on \(R\) is defined by the corresponding formulas for
\(a_n\) and \(b_r\) in the variables \(p_n\) and \(\widetilde p_m\).
\end{prop}

\begin{proof}
On \(R\), the operators are
\[
a_{-n}=(-1)^{n-1}\frac{1}{\beta}p_n,
\qquad
 a_n=(-1)^{n-1}\beta n\frac{\partial}{\partial p_n},
\]
\[
b_{-(m+1/2)}=(-1)^m\widetilde p_m,
\qquad
 b_{m+1/2}=(-1)^m\frac{\partial}{\partial\widetilde p_m}.
\]
The identities \(\Phi(p_n)=n!x_n\) and
\(\Phi(\widetilde p_m)=m!y_{m+1/2}\) carry these operators to the corresponding
operators on \(\mathcal F\).  Since \(L_n\) and \(G_r\) are polynomials in these
operators with the same normal ordered formulas, \(\Phi\) intertwines the
actions.
\end{proof}

\begin{cor}\label{cor:singular-pullback-paper}
An element \(P\in R\) is singular if and only if \(\Phi(P)\) is singular in
\(\mathcal F\).
\end{cor}

\begin{proof}
This follows immediately from \LaterProp{prop:Phi-intertwines-paper}.
\end{proof}

We use the following basis of the algebra of symmetric functions.  Define
\(e_m(p)\) by
\[
\sum_{m\ge0}e_m(p)t^m
=
\exp\left(
\sum_{n\ge1}(-1)^{n-1}\frac{p_nt^n}{n}
\right).
\]
For a partition \(\lambda=(\lambda_1,\ldots,\lambda_k)\), put
\[
\mathsf{s}_\lambda(p)
=
\det(e_{\lambda_a+b-a}(p))_{1\le a,b\le k}.
\]
Equivalently, \(\mathsf{s}_\lambda(p)=s_{\lambda'}(p)\) in the usual notation
for Schur functions.  This convention is chosen because the Giambelli formula in
Section~4 is written in terms of Chern classes.

For a finite subset \(J=\{j_1<\cdots<j_m\}\subset\Z_{\ge0}\), set
\[
q_J=\prod_{\nu=1}^m\widetilde p_{j_\nu},
\qquad
q_\varnothing=1.
\]
On the odd factor, we use the generalized odd Koschorke class
\[
[\Sigma_J^{\mathrm{odd}}]
=
c_{j_1+1/2}\cdots c_{j_m+1/2},
\qquad
[\Sigma_\varnothing^{\mathrm{odd}}]=1.
\]

\begin{prop}\label{prop:koschorke-basis-paper}
The elements
\[
\mathsf{s}_\lambda(p)q_J
\]
form a homogeneous \(\C\)-basis of \(R\).  Under \(\Phi\), this basis is carried to
\[
[\Sigma_\lambda]\otimes[\Sigma_J^{\mathrm{odd}}]
\]
in \(H^*(F_{\mathrm{ev}}\times F_{\mathrm{odd}};\C)\).  Equivalently,
\[
\Phi(\mathsf{s}_\lambda(p)q_J)
=
[\Sigma_\lambda]\otimes[\Sigma_J^{\mathrm{odd}}].
\]
\end{prop}

\begin{proof}
Since conjugation \(\lambda\mapsto\lambda'\) is a bijection on partitions, the
functions \(\mathsf{s}_\lambda(p)=s_{\lambda'}(p)\) form a basis of the symmetric
function algebra \(\C[p_1,p_2,\ldots]\).  The monomials \(q_J\) form a basis of
the exterior algebra.  Hence their products form a homogeneous basis of \(R\).

The defining identity for \(e_m(p)\) gives, after substituting \(p_n=n!x_n\),
\[
e_m(p)\longmapsto c_m.
\]
Therefore
\[
\Phi(\mathsf{s}_\lambda(p))
=
\det(c_{\lambda_a+b-a})_{1\le a,b\le k}
=[\Sigma_\lambda]
\]
by \LaterCor{cor:generalized-koschorke-giambelli-paper}.  Moreover,
\[
\Phi(q_J)
=
\prod_{\nu=1}^m j_\nu!\,y_{j_\nu+1/2}
=
 c_{j_1+1/2}\cdots c_{j_m+1/2}
= [\Sigma_J^{\mathrm{odd}}].
\]
Multiplicativity of \(\Phi\) proves the formula.
\end{proof}

\begin{rem}
This is the point at which arbitrary partitions are needed.  The Koschorke classes do not form a basis indexed by all partitions, whereas the
classes \([\Sigma_\lambda]\) do.
\end{rem}

%%%%%%%%%%%%%%%%%%%%%%%%%%%%%%%%%%%%%%%%%%%%%%%%
\subsection{Singular vectors and expansions in generalized Koschorke classes}
%%%%%%%%%%%%%%%%%%%%%%%%%%%%%%%%%%%%%%%%%%%%%%%%

\begin{thm}\label{thm:super-singular-expansion-paper}
Let \(P\in R\) be a singular vector for the polynomial representation of the
Neveu--Schwarz algebra defined above.  Then there are unique coefficients
\(a_{\lambda,J}\in\C\), only finitely many nonzero, such that
\[
P=\sum_{\lambda,J}a_{\lambda,J}\mathsf{s}_\lambda(p)q_J.
\]
Consequently,
\[
\Phi(P)
=
\sum_{\lambda,J}a_{\lambda,J}
[\Sigma_\lambda]\otimes[\Sigma_J^{\mathrm{odd}}]
\in H^*(F_{\mathrm{ev}}\times F_{\mathrm{odd}};\C).
\]
\end{thm}

\begin{proof}
The basis expansion follows from \LaterProp{prop:koschorke-basis-paper}.  Since
\(P\) is a polynomial, only finitely many coefficients occur.  Applying \(\Phi\)
and using the same proposition gives the displayed cohomological formula.
\end{proof}

\begin{rem}
\LaterThm{thm:super-singular-expansion-paper} is a statement about expansion in a chosen polynomial representation.  It does not assert that the individual
basis elements \([\Sigma_\lambda]\otimes[\Sigma_J^{\mathrm{odd}}]\) are singular.
It says that, after a singular vector is chosen, its expansion in this basis is
unique.  No closed formula for these coefficients is
asserted here.  The general Jack-superpolynomial coefficient formulas proposed by
Desrosiers--Lapointe--Mathieu \cite{DLM} are conjectural; moreover, they concern
expansions in a different basis.  Even assuming those formulas, comparison with
the present coefficients would require the transition matrix from that basis to
the Schur--exterior basis.
\end{rem}


\begin{thebibliography}{99}

\bibitem{AbrahamRobbin}
R. Abraham and J. Robbin,
\emph{Transversal Mappings and Flows},
An appendix by A. Kelley,
W. A. Benjamin, Inc., New York--Amsterdam, 1967, x+161 pp.

\bibitem{AtiyahJones}
M. F. Atiyah and J. D. S. Jones,
Topological aspects of Yang--Mills theory,
\emph{Comm. Math. Phys.} \textbf{61} (1978), no. 2, 97--118.

\bibitem{AtiyahKTheory}
M. F. Atiyah,
\emph{K-Theory},
Notes by D. W. Anderson, Second edition,
Adv. Book Classics,
Addison-Wesley Publishing Company, Advanced Book Program, Redwood City, CA, 1989, xx+216 pp.

\bibitem{AtiyahHilbert}
M. F. Atiyah,
Algebraic topology and operators in Hilbert space,
in \emph{Lectures in Modern Analysis and Applications, I},
Lecture Notes in Math., Vol. 103,
Springer-Verlag, Berlin--New York, 1969, 101--121.

\bibitem{AtiyahPatodiSingerIII}
M. F. Atiyah, V. K. Patodi and I. M. Singer,
Spectral asymmetry and Riemannian geometry. III,
\emph{Math. Proc. Cambridge Philos. Soc.} \textbf{79} (1976), no. 1, 71--99.

\bibitem{AtiyahSegalTwisted}
M. F. Atiyah and G. Segal,
Twisted K-theory,
\emph{Ukr. Mat. Visn.} \textbf{1} (2004), no. 3, 287--330;
translation in \emph{Ukr. Math. Bull.} \textbf{1} (2004), no. 3, 291--334.

\bibitem{AtiyahSegal}
M. F. Atiyah and G. Segal,
Twisted K-theory and cohomology,
in \emph{Inspired by S. S. Chern},
Nankai Tracts Math., 11,
World Scientific Publishing Co. Pte. Ltd., Hackensack, NJ, 2006, 5--43.

\bibitem{AtiyahSingerSkew}
M. F. Atiyah and I. M. Singer,
Index theory for skew-adjoint Fredholm operators,
\emph{Inst. Hautes \'Etudes Sci. Publ. Math.} No. 37 (1969), 5--26.

\bibitem{Borel}
A. Borel,
Sur la cohomologie des espaces fibr\'es principaux et des espaces homog\`enes de groupes de Lie compacts,
\emph{Ann. of Math. (2)} \textbf{57} (1953), 115--207.

\bibitem{Cibotaru}
D. Cibotaru,
The odd Chern character and index localization formulae,
\emph{Comm. Anal. Geom.} \textbf{19} (2011), no. 2, 209--276.

\bibitem{DLM}
P. Desrosiers, L. Lapointe and P. Mathieu,
Superconformal field theory and Jack superpolynomials,
\emph{J. High Energy Phys.} \textbf{2012}, no. 9, 037, front matter+41 pp.

\bibitem{DiFrancescoMathieuSenechal}
P. Di Francesco, P. Mathieu and D. S\'en\'echal,
\emph{Conformal Field Theory},
Grad. Texts Contemp. Phys.,
Springer-Verlag, New York, 1997, xxii+890 pp.

\bibitem{EellsAlexanderPontrjagin}
J. Eells, Jr.,
Alexander--Pontrjagin duality in function spaces,
in \emph{Differential Geometry},
Proc. Sympos. Pure Math., Vol. III,
American Mathematical Society, Providence, RI, 1961, 109--129.

\bibitem{Fulton}
W. Fulton,
\emph{Young Tableaux},
With applications to representation theory and geometry,
London Math. Soc. Stud. Texts, 35,
Cambridge University Press, Cambridge, 1997, x+260 pp.

\bibitem{GomiSugaku}
K. Gomi,
Twisted K-theory [translation of MR2918235],
\emph{Sugaku Expositions} \textbf{28} (2015), no. 2, 157--187.
Translation of \emph{S\=ugaku} \textbf{64} (2012), no. 1, 47--74.

\bibitem{GomiMickelsson}
K. Gomi,
Mickelsson's twisted K-theory invariant and its generalizations,
preprint, arXiv:1303.5159 [math.AT], 2013.

\bibitem{IoharaKogaI}
K. Iohara and Y. Koga,
Representation theory of Neveu--Schwarz and Ramond algebras. I. Verma modules,
\emph{Adv. Math.} \textbf{178} (2003), no. 1, 1--65.

\bibitem{IoharaKogaII}
K. Iohara and Y. Koga,
Representation theory of Neveu--Schwarz and Ramond algebras. II. Fock modules,
\emph{Ann. Inst. Fourier (Grenoble)} \textbf{53} (2003), no. 6, 1755--1818.

\bibitem{KacWakimoto}
V. G. Kac and M. Wakimoto,
Modular invariant representations of infinite-dimensional Lie algebras and superalgebras,
\emph{Proc. Nat. Acad. Sci. U.S.A.} \textbf{85} (1988), no. 14, 4956--4960.

\bibitem{Koschorke}
U. Koschorke,
Infinite-dimensional K-theory and characteristic classes of Fredholm bundle maps,
in \emph{Global Analysis} (Proc. Sympos. Pure Math., Vols. XIV, XV, XVI, Berkeley, Calif., 1968), 95--133,
Proc. Sympos. Pure Math., XIV--XVI,
American Mathematical Society, Providence, RI, 1970.

\bibitem{Kuiper}
N. H. Kuiper,
The homotopy type of the unitary group of Hilbert space,
\emph{Topology} \textbf{3} (1965), 19--30.

\bibitem{LangFDG}
S. Lang,
\emph{Fundamentals of Differential Geometry},
Grad. Texts in Math., 191,
Springer-Verlag, New York, 1999, xviii+535 pp.

\bibitem{Macdonald}
I. G. Macdonald,
\emph{Symmetric Functions and Hall Polynomials},
Second edition. With contribution by A. V. Zelevinsky and a foreword by Richard Stanley,
Oxford Mathematical Monographs,
Oxford University Press, Oxford, 1995.

\bibitem{Milnor}
J. Milnor,
On axiomatic homology theory,
\emph{Pacific J. Math.} \textbf{12} (1962), 337--341.

\bibitem{MimachiYamada}
K. Mimachi and Y. Yamada,
Singular vectors of the Virasoro algebra in terms of Jack symmetric polynomials,
\emph{Comm. Math. Phys.} \textbf{174} (1995), no. 2, 447--455.

\bibitem{Morimoto}
H. Morimoto,
Infinite-dimensional cycles associated to operators,
\emph{Nagoya Math. J.} \textbf{127} (1992), 1--14.

\bibitem{Nakajima}
H. Nakajima,
\emph{Lectures on Hilbert Schemes of Points on Surfaces},
Univ. Lecture Ser., 18,
American Mathematical Society, Providence, RI, 1999, xii+132 pp.

\bibitem{Palais}
R. S. Palais,
Homotopy theory of infinite dimensional manifolds,
\emph{Topology} \textbf{5} (1966), 1--16.

\bibitem{Phillips}
J. Phillips,
Self-adjoint Fredholm operators and spectral flow,
\emph{Canad. Math. Bull.} \textbf{39} (1996), no. 4, 460--467.

\bibitem{QuillenCayley}
D. Quillen,
Superconnection character forms and the Cayley transform,
\emph{Topology} \textbf{27} (1988), no. 2, 211--238.

\bibitem{Smale}
S. Smale,
An infinite dimensional version of Sard's theorem,
\emph{Amer. J. Math.} \textbf{87} (1965), 861--866.

\bibitem{TsuchiyaKanie}
A. Tsuchiya and Y. Kanie,
Fock space representations of the Virasoro algebra. Intertwining operators,
\emph{Publ. Res. Inst. Math. Sci.} \textbf{22} (1986), no. 2, 259--327.

\bibitem{WakimotoYamadaI}
M. Wakimoto and H. Yamada,
Irreducible decompositions of Fock representations of the Virasoro algebra,
\emph{Lett. Math. Phys.} \textbf{7} (1983), no. 6, 513--516.

\bibitem{WakimotoYamadaII}
M. Wakimoto and H. Yamada,
The Fock representations of the Virasoro algebra and the Hirota equations of the modified KP hierarchies,
\emph{Hiroshima Math. J.} \textbf{16} (1986), no. 2, 427--441.

\bibitem{Whitehead}
G. W. Whitehead,
\emph{Elements of Homotopy Theory},
Grad. Texts in Math., 61,
Springer-Verlag, New York--Berlin, 1978, xxi+744 pp.

\end{thebibliography}
\end{document}